\documentclass[a4paper, 11pt,reqno]{amsart}

\usepackage[utf8]{inputenc}
\usepackage[T1]{fontenc}
\usepackage{mathtools, amsthm, amsfonts, amssymb, amsmath}
\usepackage{microtype}

\usepackage{latexsym}
\usepackage{braket}
\usepackage{thmtools}
\usepackage{mathrsfs}
\usepackage{textcomp}
\usepackage{graphicx}
\usepackage{setspace}
\usepackage{nicefrac}
\usepackage{enumerate}
\usepackage{wasysym}
\usepackage[normalem]{ulem}
\usepackage{upgreek}

\usepackage{paralist}
\usepackage{xcolor}
\usepackage{etoolbox}
\usepackage{mathdots}
\usepackage{wrapfig}
\usepackage{floatflt}
\usepackage{tensor} 
\usepackage{lscape}
\usepackage{stmaryrd}
\usepackage[enableskew]{youngtab}
\usepackage{ytableau}
\usepackage{pifont}

\usepackage{tikz}
\usepackage{tikz-cd}
\usepackage{tikz-3dplot}
\usetikzlibrary{arrows}
\usetikzlibrary{matrix}
\usetikzlibrary{positioning}
\usetikzlibrary{calc,fit,backgrounds,decorations.pathmorphing}
\usetikzlibrary{cd}

\usepackage{hyperref}
\hypersetup{
    colorlinks,
  linkcolor=[rgb]{0.3,0.3,0.6},
  citecolor=[rgb]{0.2, 0.6, 0.2},
  urlcolor=[rgb]{0.6, 0.2, 0.2}
}

\usepackage{enumitem}

\newcommand{\vvirg}{ , \dots , }
\newcommand{\ootimes}{ \otimes \cdots \otimes }

\newcommand{\ttimes}{ \times \cdots \times }

\newcommand{\contract}{\rotatebox[origin=c]{180}{ \reflectbox{$\neg$} }}

\newcommand{\bigboxplus}{{\scalebox{1.2}{$\boxplus$}}}

\newcommand{\bfI}{\mathbf{I}}

\newcommand{\bfe}{\mathbf{e}}

\newcommand{\bfm}{\mathbf{m}}
\newcommand{\bfn}{\mathbf{n}}

\newcommand{\bft}{\mathbf{t}}

\newcommand{\bfv}{\mathbf{v}}

\newcommand{\bfz}{\mathbf{z}}

\newcommand{\calC}{\mathcal{C}}

\newcommand{\calT}{\mathcal{T}}

\newcommand{\calZ}{\mathcal{Z}}

\newcommand{\scrE}{\mathscr{E}}

\newcommand{\bbC}{\mathbb{C}}

\newcommand{\bbP}{\mathbb{P}}

\newcommand{\frakg}{\mathfrak{g}}

\newcommand{\bfzeta}{\boldsymbol{\zeta}}

\renewcommand{\phi}{\varphi}
\newcommand{\eps}{\varepsilon}

\renewcommand{\tilde}[1]{\widetilde{#1}}

\renewcommand{\bar}[1]{\overline{#1}}

\newcommand{\id}{\mathrm{id}}

\newcommand{\rank}{\mathrm{rank}}
\newcommand{\im}{\mathrm{im} \;}  
\newcommand{\image}{\mathrm{image}}

\DeclareMathOperator{\Hom}{Hom}
\DeclareMathOperator{\Gr}{Gr}

\DeclareMathOperator{\End}{End}

\DeclareMathOperator{\ann}{\mathfrak{ann}}

\newcommand{\expdim}{\mathrm{expdim}}

\DeclareMathAccent{\wtilde}{\mathord}{largesymbols}{"65}

\newcommand{\GL}{\mathrm{GL}}
\newcommand{\PGL}{\mathrm{PGL}}
\newcommand{\SL}{\mathrm{SL}}
\newcommand{\fraksl}{\mathfrak{sl}}
\newcommand{\frakgl}{\mathfrak{gl}}

\newcommand{\Mat}{\mathrm{Mat}}

\DeclareMathOperator{\Stab}{Stab}

\newcommand{\TNS}{\mathcal{T\!N\!S}}
\newcommand{\calTNS}{\mathcal{T\!N\!S}}

\usepackage[top=3cm, bottom=3cm, left=3cm, right =3cm]{geometry}

\numberwithin{equation}{section}

\newtheorem{theorem}[equation]{Theorem}
\newtheorem{lemma}[theorem]{Lemma}

\newtheorem{corollary}[theorem]{Corollary}

\theoremstyle{definition}
\newtheorem{definition}[theorem]{Definition}

\newtheorem*{question*}{Question}

\newtheorem{remarkx}[theorem]{Remark}
\newenvironment{remark}
  {\pushQED{\qed}\remarkx}
  {\popQED\endremarkx}

\newtheorem{examplex}[theorem]{Example}
\newenvironment{example}
  {\pushQED{\qed}\examplex}
  {\popQED\endexamplex}

\title[Triangular tensor networks]{Triangular tensor networks, \\ 
pencils of matrices and beyond}

\author[A. Bernardi]{Alessandra Bernardi}
\address[A. Bernardi]{Dipartimento di Matematica; Università di Trento;  I-38123 Trento, Italy}
\email{alessandra.bernardi@unitn.it}

\author[F. Gesmundo]{Fulvio Gesmundo}
\address[F. Gesmundo]{Institut de Mathématiques de Toulouse; UMR5219 -- Université de Toulouse; CNRS -- UPS, F-31062 Toulouse Cedex 9, France}
\email{fgesmund@math.univ-toulouse.fr; fulges.math@gmail.com}

\keywords{Tensor networks, graph tensors, pencils of matrices}

\subjclass[2020]{15A69, 81P45, 20G05, 14M12}

\begin{document}

\begin{abstract}
We study tensor network varieties associated with the triangular graph, with a focus on the case where one of the physical dimensions is $2$. This allows us to interpret the tensors as pencils of matrices. We provide a complete characterization of these varieties in terms of the Kronecker invariants of pencils. We determine their dimension, identifying the cases for which the dimension is smaller than the expected parameter count. We provide necessary conditions for membership in these varieties, in terms of the geometry of classical determinantal varieties, coincident root loci and plane cubic curves. We address some extensions to arbitrary graphs. 
\end{abstract}

\maketitle

\section{Introduction}

Tensor network states provide a natural geometric model for describing tensors arising from quantum many-body systems and have become a central object of study in algebraic geometry, representation theory, and quantum information theory, see, e.g., \cite{VerstraeteWolfPerezGarciaCirac2006PEPS,CLSWWs, Orus2014PracticalIntroTN}. From the physical perspective, tensor network states offer efficient ansatz classes for approximating ground states of local Hamiltonians, see \cite{White1992DMRG, Schollwoeck2011DMRG_MPS}, with the graph topology modeling the geometry of the system and a set of weights, called bond dimensions, controlling the expressive power of the class. From the mathematical perspective, the set of tensors representable by a given network naturally yields an algebraic variety whose geometry encodes fundamental structural properties of the model, see \cite{LanQiYe,ChrLucVraWer}.

A general upper bound for the dimension of tensor network varieties was established in \cite{BDG23}, refining earlier results in \cite{PerVerWolCir,VerMurCir,LanQiYe}. The upper bound is known to be sharp for several families of graphs and bond dimension configurations, such as matrix product states with open boundary conditions \cite{HaegMarOsbVer}, certain tree networks \cite{BucBucMic}, and for certain choices of bond dimensions \cite{BDG23}. In physical terms, this means that the network parametrizes a family of states as large as predicted by the na\"ive parameter count. However, there are explicit configurations in which the dimension falls short of the expected value as observed in \cite[Sec. 5]{BDG23}; this phenomenon in geometry takes the name of \emph{defectivity}, see \cite{CC02, Bernardi2018}.

Defective tensor network varieties correspond to networks whose expressive power is strictly smaller than what one would expect from a parameter count. Understanding when this situation occurs is important both from a mathematical and physical point of view \cite{BarthelLuFriesecke2022}. On the mathematical side, defectivity is a subtle property which often hides interesting geometric features; for instance, in the study of secant varieties, defectivity is closely related to classical questions about degeneracy loci and their geometric properties \cite{Terr21}; see also \cite[Ch. 3]{Russo} for an overview. On the physical side, defectivity yields a reduction in the set of states accessible to the ansatz class, which can have direct implications both for the effectiveness of the tensor network ansatz and for the dimensionality of the algorithms to approximate a given state \cite{Schollwoeck2011DMRG_MPS,GersterRizziSilviDalmonteMontangero2017FQH_Hofstadter_TN}.

Among all graphs, the triangular graph is the smallest non-trivial topology where defectivity can appear. The triangle has already played a key role in the analysis of networks in \cite{BDG23}, where the first defective configurations were identified in small dimension. Triangular configurations also appear as fundamental building blocks in higher-dimensional networks, such as PEPS on two-dimensional triangular lattices \cite{Orus2014PracticalIntroTN,ChrLucVraWer}, where local defects may propagate to the global tensor network variety. Finally, the triangular tensor network has connections with the geometry of the orbit-closure of the matrix multiplication tensor, a central object in algebraic complexity theory, see, e.g., \cite{ChrVraZui}.
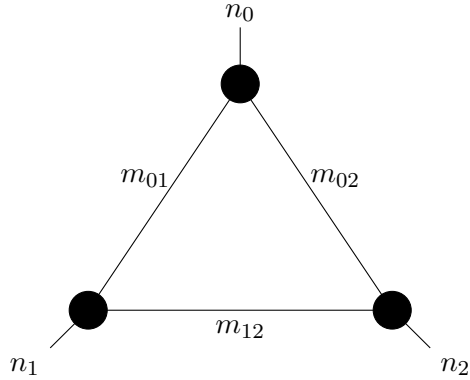
\begin{figure}[!h]
\centering
\begin{tikzpicture}[scale=.5]
\draw (0,0)--(4,6);
\draw (0,0)--(8,0);
\draw (8,0)--(4,6);
\node () at (1.5,3.5) {$m_{01}$};
\node () at (6.5,3.5) {$m_{02}$};
\node () at (4,-0.5) {$m_{12}$};
\draw (0,0)--(-1,-1);
\draw (4,6)--(4,7.5);
\draw (8,0)--(9,-1);
\draw[fill=black] (0,0) circle (.5cm);
\draw[fill=black] (4,6) circle (.5cm);
\draw[fill=black] (8,0) circle (.5cm);
\node[anchor = north east] at (-1,-1){$n_1$};
\node[anchor = north west] at (9,-1){$n_2$};
\node[anchor = south] at (4,7.4){$n_0$};
\node[anchor = north east] at (-1,-1){\phantom{$\bbC^{n_1}$}};
\node[anchor = north west] at (9,-1){\phantom{$\bbC^{n_2}$}};
\node[anchor = south] at (4,7.4){\phantom{$\bbC^{n_0}$}};
\node at (0,0){$1$};
\node at (4,6){$2$};
\node at (8,0){$0$};
\node at (-2,2){\phantom{$T = $}};
\end{tikzpicture}
\caption{Triangular tensor network with bond dimensions $\bfm = (m_{01},m_{12},m_{02})$ and local dimensions $\bfn = (n_0,n_1,n_2)$.} \label{fig: triangle graph}
\end{figure}
In this work, we give a full characterization of the tensor network varieties in the case of the triangle graph, with one physical dimension equal to $2$. This is the case classically known as \emph{pencils of matrices}: a tensor $T \in \bbC^2 \otimes \bbC^{n_1} \otimes \bbC^{n_2}$ can be identified with a matrix of size $n_1 \times n_2$ whose entries are linear polynomials in two variables.

Our work provides the first complete characterization of the elements in an infinite family of tensor network varieties, including the computation of the dimension of such varieties and of normal forms for its elements. Further, we determine families of equations vanishing on the variety, providing necessary membership conditions. Our result is summarized in the following theorem.

\begin{theorem}\label{main theorem}
Let $\bfn = (2,n_1,n_2)$ and $\bfm = (m_{01} , m_{12}, m_{02})$ be an assignment of local and physical dimensions of the triangular network $\Delta$, with $n_i \leq m_{ij}m_{ik}$ with $\{i,j,k\} = \{0,1,2\}$. Write $k_i = m_{ij}m_{ik} - n_i$ for $i = 1,2$ and suppose $k_2 \leq k_1$. If 
\begin{equation}\label{eqn: nontriviality condition} 
    \tag{\raisebox{0.3mm}{\textreferencemark}}
m_{01} = m_{02} \quad \text{ and } \quad  k_2 \leq k_1 < m_{12} 
\end{equation}
set $m=m_{01} = m_{02}$. The following holds:
\begin{enumerate}[leftmargin=*]
\item\label{Part1} \underline{Dimension}. The affine dimension of the tensor network variety $\calTNS^\Delta_{\bfm,\bfn}$ is 
    \[
 \dim \calTNS^{\Delta}_{\bfm,\bfn} = 2 m^2 m_{12}^{2} - m m_{12}^{2} - m m_{12} k_{1} - m m_{12} k_{2} - m k_{1} k_{2} + 2 k_{1} k_{2} + m .
    \]
    The difference between the ``expected dimension'' of $\calTNS^{\Delta}_{\bfm,\bfn}$ and the value above is 
    \[
   \delta^{\Delta}_{\mathbf m,\mathbf n}= (m-2)k_{1}k_{2}+(m-1)(m_{12}^{2}-1)
    \]
    unless $\calTNS^{\Delta}_{\bfm,\bfn}$ is expected to fill the ambient space, in which case it is 
    \[
    \delta^{\Delta}_{\mathbf m,\mathbf n} = m (m_{12}^{2}-m_{12}(k_{1}+k_{2})+k_{1}k_{2}-1).
    \]
\item\label{Part2} \underline{Normal form}. The generic element $T \in \calTNS^{\Delta}_{\bfm,\bfn} \subseteq \bbC^2 \otimes \bbC^{n_1} \otimes \bbC^{n_2}$, regarded as $n_1 \times n_2$ matrix of linear forms in two variables $v_0,v_1$ has the form
\[
T = \left( \begin{array}{cc}
    T_1 & 0 \\
    0 & T_2
\end{array} \right)
\]
where 
\begin{align*}
T_1 = &\left[ 
\begin{array}{c | c | c}
(v_0 + \xi_1 v_1) \cdot \id_{m_{12}-k_{1}} & & \\
\hline
& \ddots & \\
\hline
& & (v_0 + \xi_m v_1) \cdot \id_{m_{12} -k_{1}}
\end{array}\right], \\
T_2 &\text{ generic in size $ [(m-1)k_{1}] \times [(m-1) k_{1} + (k_{1}-k_{2})]$}
\end{align*}
where $\xi_1 \vvirg \xi_m$ are generic complex numbers. 
\item\label{Part3} \underline{Equations}. A system of equations for the variety $\calTNS^{\Delta}_{\bfm,\bfn}$ arises as pullback of the equations of the coincident root locus of multiplicity 
$\lambda(\kappa) = ((m_{12}-\kappa )^{m} , 1^{(m-1)\kappa} )$ for $\kappa = k_1 \vvirg m_{12} - 2$.
\end{enumerate}
If \eqref{eqn: nontriviality condition} does not hold, then $\calTNS^{\Delta}_{\bfm,\bfn} = \bbC^2 \otimes \bbC^{n_1} \otimes \bbC^{n_2}$ and $\dim \calTNS^{\Delta}_{\bfm,\bfn} = 2 n_1 n_2$.
\end{theorem}

The characterization in  \autoref{main theorem}, \autoref{Part2}, shows that a generic tensor in $\TNS^\Delta_{\bfm,\bfn} \subseteq \mathbb{C}^2\otimes\mathbb{C}^{n_1}\otimes\mathbb{C}^{n_2}$, viewed as an $n_1\times n_2$ pencil, has a specific block decomposition: one block is a direct sum of several copies of the same full-rank pencil, while the complementary block is a generic pencil of smaller size. The dimension formula in \autoref{main theorem}, \autoref{Part1} is a direct consequence of this normal form. The equations in \autoref{main theorem}, \autoref{Part3} are obtained by imposing that the repeated block has coincident generalized eigenvalues, according to the structure of \autoref{Part2}. 

In the special case where $n_2 = n_1+1$, the result on the normal form from \autoref{main theorem}, \autoref{Part2} guarantees that the tensor network variety $\calTNS^\Delta_{\bfm,\bfn}$ is contained in the nullcone for the action of $\SL_2 \times \SL_{n_1} \times \SL_{n_2}$ on $\bbC^2 \otimes \bbC^{n_1} \otimes \bbC^{n_2}$. Using this fact, we can partially recover some of the results of \cite{vdBCLNZ25} about non-membership of certain points in the moment polytope of the matrix multiplication tensor. We discuss this in \autoref{sec:4}.

We emphasize that the subcriticality assumption $n_i\le m_{ij}m_{ik}$ is not restrictive, see \autoref{subsec: graph tensors}. By contrast, the hypothesis \eqref{eqn: nontriviality condition} is essential: if it fails, then the tensor network variety fills the ambient space. When $m_{01}\neq m_{02}$ this is proved in \autoref{thm: different bond dimensions} whereas when $k_1\ge m_{12}$ it follows from \autoref{thm: triangular subcritical}; in particular, for $k_1=m_{12}$ the normal form degenerates to a generic tensor in $\mathbb{C}^2\otimes\mathbb{C}^{n_1}\otimes\mathbb{C}^{n_2}$.

We list some additional contributions of the paper.
\begin{itemize}
  \item We provide a set of equations for $\calTNS^\Delta_{\bfm,\bfn}$ with one physical dimension equal to $2$ as pullback of equations of coincident root loci. See \autoref{thm: equations from coincident root loci}.
  \item We provide a geometric interpretation of the results of \autoref{main theorem} in terms of intersection theoretic properties of the pencil in the special case $\bfm = (2,m,2)$. This provides a complete system of set-theoretic equations for $\calTNS^\Delta_{\bfm,\bfn}$ and recovers the results of \cite[Theorem 5.2]{BDG23}. See \autoref{thm: triangle and Z variety}.
  \item We provide a geometric characterization of $\calTNS^\Delta_{(2,2,2),(3,3,3)}$. This yields a complete system of set-theoretic equations for the tensor network variety. See \autoref{cor: ruppert for TNS} and \autoref{thm: triangle 222 333}. 
   \item We provide a partial generalization of \autoref{main theorem} to arbitrary networks, with special configuration of physical dimensions. See \autoref{augmentation:theorem} and \autoref{cor:defect_growth_many_pins}.
\end{itemize}

\paragraph{\textbf{Outline.}}
In \autoref{section:preliminaries} we introduce some basics on tensor networks, we define the expected dimension and the defect, record the Kronecker classification of matrix pencils, as well as some results on symmetry Lie algebra of tensors under restricted actions. The proof of \autoref{main theorem}, \autoref{Part1}, and \autoref{main theorem}, \autoref{Part2} are contained in \autoref{section: pencils}, which establishes the dimension and normal form results for $\TNS^{\Delta}_{\mathbf m,\mathbf n}$. In \autoref{sec:4}, we describe a system of equations for $\TNS^{\Delta}_{\mathbf m,\mathbf n}$ via coincident root loci, proving \autoref{main theorem}, \autoref{Part3}. \autoref{sec:5} places the result of \autoref{main theorem} into a more general geometric construction; we use it to characterize the tensor network variety $\calTNS^\Delta_{(2,2,2),(3,3,3)}$ as well. \autoref{sec:6} studies extensions to arbitrary graphs via augmentation with vertices of local dimension $2$. 

\section{Preliminaries}\label{section:preliminaries}

Throughout the paper, all vector spaces are complex vector spaces of finite dimension; for a space $V$, let $V^*$ denote its dual space. A \emph{variety} is an affine or projective complex algebraic variety, defined as the vanishing set of a system of polynomial equations. 

The tensor space $V_1 \ootimes V_d$ of tensors of order $d$ is the space of multilinear maps $V_1^* \ttimes V_d^* \to \bbC$. The group $\GL(V_1) \ttimes \GL(V_d)$ naturally acts on $V_1 \ootimes V_d$. This action extends to an action of $\End(V_1) \ttimes \End(V_d)$. Given two tensors $T,S \in V_1 \ootimes V_d$, we say that $S$ is a \emph{restriction} of $T$ if
\[
S \in (\End(V_1) \ttimes \End(V_d)) \cdot T
\]
namely if there exist $X_k \in \End(V_k)$ for $k = 1 \vvirg d$ such that $S = (X_1 \ootimes X_d)(T)$. We say that $S$ is a \emph{degeneration} of $T$ if $S$ is limit of restrictions of $T$.

For every subset $I \subseteq [d]$, a tensor defines a \emph{flattening map}, which is the linear map 
\[
\bigotimes_{i \in I} V_i^* \to \bigotimes_{j \notin I} V_j
\]
defined by tensor contraction. A tensor is \emph{concise on the $i$-th factor} if the flattening map $T_i: V_i^* \to \bigotimes_{j \neq i} V_j$ is injective; we say that it is \emph{concise} if it is concise on every factor. The flattening $T_1$ parameterizes a linear subspace of $V_2 \ootimes V_d$ of dimension at most $\dim V_1$. The tensor $T$ is uniquely determined, up to the action of $\GL(V_1)$, by such a subspace.

We briefly recall the basic notation and definitions for tensor network states, following the conventions of \cite{BDG23}.
\subsection{Graph tensors and tensor network states}\label{subsec: graph tensors}

Let $\Gamma$ be a finite connected graph with $d$ vertices $\mathbf{v}(\Gamma) = \{ 1 \vvirg d\}$ and set of edges $\mathbf{e}(\Gamma)$. Let $\mathbf{n} = (n_i : i \in \mathbf{v}(\Gamma))$ and $\mathbf{m} = (m_e : e \in \mathbf{e}(\Gamma))$ be collections of integer weights on the vertices and the edges of $\Gamma$, called, respectively, \emph{physical dimensions} and \emph{bond dimensions}. Associate to every vertex $i \in \bfv(\Gamma)$ a \emph{physical space} $V_i \simeq \mathbb{C}^{n_i}$. Assign an arbitrary orientation to the edges of $\Gamma$, and for every edge $e = (i_1,i_2)$ consider a pair of \emph{bond spaces} $E_e, E_e^*$, dual to each other with $\dim E_e = m_e$. Let $W_i = \bigotimes_{e \ni i} E'_e$ where $E'_e = E_e$ if $i$ is the source of $e$ and $E'_e = E_e^*$ if $i$ is the target of $e$.

The \emph{graph tensor} of $(\Gamma, \bfm)$ is the tensor 
\[
T(\Gamma,\bfm) = \bigotimes_{e \in \bfe(\Gamma)} \id_{E_e}
\]
where $\id_{E_e} \in E_e \otimes E_e^*$ is the identity map from $E_e$ to itself. After regrouping the factors as prescribed by the structure of $\Gamma$, the graph tensor
\[
T(\Gamma,\bfm) \in \bigotimes_{e \in \bfe(\Gamma)} (E_e^* \otimes E_e) \simeq \bigotimes_{i \in \bfv(\Gamma)} \Bigl[ \bigotimes_{e \ni i} E'_e \Bigr] = \bigotimes _{i \in \bfv(\Gamma)} W_i
\]
can be regarded as a tensor of order $d$. It is easy to see that $T(\Gamma,\bfm)$ is concise in $W_1 \ootimes W_d$ and different choices of orientation yield the same graph tensor up to the action of $\GL(W_1) \ttimes \GL(W_d)$.

The tensor network variety $\calTNS^\Gamma_{\bfm,\bfn}$ is the set of all degenerations of $T(\Gamma,\bfm)$ to $V_1 \ootimes V_d$. More precisely, let $\Hom(W_i, V_i)$ be the space of linear maps from $W_i$ to $V_i$. Define 
\begin{equation}\label{eqn: map Phi}
\begin{aligned}
\Phi: \prod_{i \in \bfv(\Gamma)} \Hom(W_i, V_i) &\to V_1 \ootimes V_d \\
(X_1 \vvirg X_d) &\mapsto (X_1 \ootimes X_d)(T(\Gamma,\bfm)).
\end{aligned}
\end{equation}
The closure, equivalently taken in the Zariski or the Euclidean topology, of the image of $\Phi$, that is 
\[
\calTNS^\Gamma_{\bfm,\bfn} = \bar{ \image(\Phi) } \subseteq V_1 \ootimes V_d,
\]
is the tensor network variety on $\Gamma$ with bond dimension $\bfm$ and physical dimensions $\bfn$. By construction, $\calTNS^\Gamma_{\bfm,\bfn}$ is closed under the action of $\GL(V_1) \ttimes \GL(V_d)$.

Equivalently, one may consider \emph{local} tensors $X_i \in \Hom(W_i,V_i) \simeq \left(\bigotimes_{e\ni i}{E'}_e^* \right) \otimes V_i$ and regard the map $\Phi$ simply as the tensor contraction of $\bigotimes_{i \in \bfv(\Gamma)} X_i$ between each bond space $E_e$ with its dual space $E_e^*$. In this way $\calTNS^\Gamma_{\bfm,\bfn}$ is realized, up to closure, as the set of tensors arising via tensor contraction of local tensors. This is the point of view that is often taken in the physics literature. However, explicitly realizing the tensor network variety as the image of the map $\Phi$ via evaluation at the graph tensor makes it easier to state some of the results in the following sections.

Following \cite{LanQiYe}, we recall the following definition.
\begin{definition}\label{def:criticality}
Let $ (\Gamma, \mathbf{m}, \mathbf{n}) $ be a tensor network. A vertex $i \in \bfv(\Gamma) $ is called:
\begin{itemize}
    \item \emph{subcritical} if $ \prod_{e \ni i} m_e \geq n_i $; it is \emph{strictly subcritical} if the inequality is strict;
    \item \emph{supercritical} if $ \prod_{e \ni i} m_e \leq n_i $; it is \emph{strictly supercritical} if the inequality is strict;
    \item \emph{critical} if $ i $ is both subcritical and supercritical.
\end{itemize}
The tensor network $(\Gamma, \mathbf{m}, \mathbf{n}) $ is called (strictly) subcritical (resp. supercritical) if all its vertices are (strictly) subcritical (resp. supercritical).
\end{definition}
The condition $n_i \leq m_{ij}m_{ik}$ in the statements of \autoref{main theorem} guarantees that the network $\Delta$ with bond dimensions $\bfm$ and local dimensions $\bfn$ is subcritical. This is not restrictive. It was shown in \cite{LanQiYe} and more formally in \cite{BDG23} that tensor network varieties with supercritical vertices can be expressed as birational images of fiber bundles over suitable Grassmannians whose fibers are corresponding subcritical tensor network varieties. In particular, if $T$ is a tensor in a tensor network variety having supercritical vertices, then $T$ is not concise; in the subspace where it is concise, it can be expressed as an element of a corresponding subcritical tensor network variety. We refer to \cite[Section 4.5]{BDG23} for the details.

\subsection{Expected dimension}
A general upper bound for the dimension of $\mathcal{TNS}^\Gamma_{\mathbf{m},\mathbf{n}}$ was given in \cite[Cor.~1.4]{BDG23}. In most cases, and in particular in the cases studied in this paper, such bound is
\begin{equation} \label{eq:expdim}
\expdim \calTNS^\Gamma_{\mathbf{m},\mathbf{n}} := \ \min \left\{ \sum_{i \in \mathbf{v}(\Gamma)} \left(n_i \cdot  \prod_{e \ni i } m_e -1\right)  -  \sum_{e \in \mathbf{e}(\Gamma)} (m_e^2-1)  +  1, \prod_{i  \in \mathbf{v}(\Gamma)} n_i \right\},
\end{equation}
and we call it the \emph{expected dimension} of the tensor network variety. If $\dim \calTNS^\Gamma_{\bfm,\bfn} < \expdim \calTNS^\Gamma_{\bfm,\bfn}$, we say that the tensor network variety is \emph{defective}. The expected dimension arises as a count of the number of free parameters in the definition of $\calTNS^\Gamma_{\mathbf{m},\mathbf{n}}$: if $\calTNS^\Gamma_{\mathbf{m},\mathbf{n}}$ does not fill the ambient space, then the expected dimension is the difference between the dimension of the domain of the map $\Phi$ from \eqref{eqn: map Phi} and the dimension of the gauge group, a subgroup of $\prod_{i \in \bfv(\Gamma)} \GL(W_i)$ which is expected to control the dimension of the fibers of $\Phi$. The gauge subgroup is realized as the image of the product $\prod_{e \in \bfe(\Gamma)} \GL(E_e)$, into $\GL(W_1) \ttimes \GL(W_d)$, via the natural action of each $\GL(E_e)$ on $E_e$ and $E_e^*$, appearing in two of the factors $W_1 \vvirg W_d$. The action of $\GL(W_1) \ttimes \GL(W_d)$ on $W_1 \ootimes W_d$ then realizes the gauge subgroup as the product $\prod_{e \in \bfe(\Gamma)}\PGL(E_e)$. It was shown in \cite[Cor. 3.7]{BDG23} that it coincides with the identity component of the stabilizer of the graph tensors $T(\Gamma,\bfm)$. In particular, it controls the dimension of the tensor network variety in the supercritical range and it plays a key role in the analysis of the dimension in the subcritical range.

The difference 
\[
\delta^\Gamma_{\bfm,\bfn} := \expdim \calTNS^\Gamma_{\bfm,\bfn} - \dim \calTNS^\Gamma_{\bfm,\bfn}
\]
is called the \emph{defect} of the tensor network variety $\calTNS^\Gamma_{\bfm,\bfn}$. In \cite[Sec. 5]{BDG23}, several examples of defective configurations were provided, in the case where $\Gamma$ is the cycle graph of length $3$ and $4$. From this point of view, \autoref{main theorem} generalizes these results, characterizing dimension and defects of tensor network varieties in the case of triangle graph with one physical dimension equal to $2$.

In order to separate the intrinsic parameter count from the possible saturation of the ambient space, it is convenient to consider the following notion of \emph{fiber defect}, which is analogous to the one defined in \cite{CC02} in the setting of secant varieties.
\begin{definition}
The \emph{fiber defect} is the difference between the actual dimension of the tensor network variety and the parameter count, that is
\[
\mathrm{fiberdefect}^\Gamma_{\mathbf{m},\mathbf{n}}
:= \ \left(
\sum_{i \in \mathbf{v}(\Gamma)} \left(n_i \cdot  \prod_{e \ni i } m_e -1\right) \ - \ \sum_{e \in \mathbf{e}(\Gamma)} (m_e^2-1) \ + \ 1\right)
- \dim \calTNS^\Gamma_{\mathbf{m},\mathbf{n}} .
\]
\end{definition}
The fiber defect controls the difference in dimension between the domain of the map $\Phi$ modulo the action of the gauge subgroup, and the actual dimension of the tensor network variety. It coincides with the defect when the variety is not expected to fill the ambient space, and it allows one to keep track of the ``unexpected'' contribution of the dimension of the fiber of the map $\Phi$.

\subsection{Pencils of matrices}\label{sec: pencils}

A pencil of matrices of size $n_1 \times n_2$ is a $2$-dimensional linear subspace of $\Mat_{n_1 \times n_2}$. By extension, it is also a tensor in $\bbC^2 \otimes \bbC^{n_1} \otimes \bbC^{n_2}$ where $\Mat_{n_1 \times n_2}$ is identified with $\bbC^{n_1} \otimes \bbC^{n_2}$ and the linear space coincides with the image of the flattening map $T: (\bbC^2)^* \to \Mat_{n_1 \times n_2}$. This correspondence is one-to-one up to the action of $\GL_2$ on the factor $\bbC^2$.

It is a classical fact due to Weierstrass and Kronecker \cite{Weis67,Kron90} that pencils of matrices have normal forms: in other words, up to the action of $\GL_2 \times \GL_{n_1} \times \GL_{n_2}$, a pencil of matrices is uniquely determined by a finite set of discrete and continuous invariants, today known as Kronecker invariants. We refer to \cite[Ch. 13]{Gant59} for the theory and we record here the result. Let $a_0,a_1$ be a basis of $\bbC^2$, let $b_1 \vvirg b_{n_1}$ and $c_1 \vvirg c_{n_2}$ be bases of $\bbC^{n_1}$ and $\bbC^{n_2}$, respectively.

There are three types of \emph{indecomposable} pencils of matrices:
\begin{itemize}
\item Left singular pencils:
 \begin{align*}
L_{p} = &a_0 \otimes \textstyle \sum_1^p b_i \otimes c_i +a_1 \otimes \textstyle \sum_1^p b_i \otimes c_{i+1} = \\ 
&\left[\begin{array}{ccccc}
        a_0 & a_1 \\ 
         & \ddots & \ddots \\
         & & a_0 & a_1
        \end{array} \right]
        \in \bbC^2 \otimes \Mat_{p \times (p+1)} 
\end{align*}
\item Right singular pencils: 
\begin{align*} 
R_{p} = &a_0 \otimes \textstyle \sum_1^p b_i \otimes c_i +a_1 \otimes \textstyle \sum_1^p b_{i+1} \otimes c_{i} = \\
&\left[{\begin{array}{cccc}
        a_0 & \\
        a_1 & \ddots  & \\
        & \ddots & a_0 \\
        & &  a_1       \end{array}}\right]
        \in \bbC^2 \otimes \Mat_{(p+1) \times p} 
\end{align*}
\item Jordan pencils with generalized eigenvalue $z = \zeta_0 a_0 + \zeta_1 a_1$: 
\begin{align*} 
    J_p(z) = &z \otimes \textstyle \sum_1^p b_i \otimes c_i +a_1 \otimes \textstyle \sum_1^{p-1} b_i \otimes c_{i+1} = \\ 
    &\left[{\begin{array}{cccc}
        z & a_1 \\
        & \ddots & \ddots \\
        & & \ddots & a_1  \\
        & &  & z 
        \end{array}}\right]
        \in \bbC^2 \otimes \Mat_{p \times p} 
\end{align*}
\end{itemize}

The \emph{block sum} of two pencils $T_1 \in \bbC^2 \otimes \bbC^{n_1} \otimes \bbC^{n_2}$ and $T_2 \in \bbC^2 \otimes \bbC^{n_1'} \otimes \bbC^{n_2'}$ is obtained by projecting the direct sum $T_1 \oplus T_2$ on the diagonal $\bbC^2 \subseteq \bbC^2 \oplus \bbC^2$; in terms of matrices of linear forms, the block sum is 
\[
T_1 \boxplus T_2 = \left( \begin{array}{cc}
T_1 & 0 \\
0 & T_2
\end{array} \right) \in \bbC^2 \otimes (\bbC^{n_1} \oplus \bbC^{n_1'})  \otimes (\bbC^{n_2} \oplus \bbC^{n_2'}).
\]
For a pencil $T$ and a matrix $M$, we also use the notation $T \boxtimes M$ to indicate the Kronecker product of the tensor $T$ and the matrix $M$ regarded as a tensor $e_0 \otimes M$. In particular, if $M$ is an identity matrix of size $q$, $T \boxtimes M$ is the block sum of $q$ copies of $T$.

The structural result of \cite{Weis67,Kron90} is that every pencil of matrices $T \in V_0 \otimes V_1 \otimes V_2$ with $\dim V_0 = 2$ is, up to the action of $\GL(V_1) \times \GL(V_2)$, a block sum of indecomposable pencils; the sizes of the blocks are uniquely determined. Moreover, the generalized eigenvalues of the Jordan blocks are uniquely determined up to the action of $\GL(V_0)$.  

We record the Kronecker structure of generic elements in $\bbC^2 \otimes \bbC^{n_1} \otimes \bbC^{n_2}$; see, e.g., \cite[Lemma A1]{GLS24}. If $n_2 > 2n_1$ or $n_1 > 2n_2$, then $\bbC^2 \otimes \bbC^{n_1} \otimes \bbC^{n_2}$ does not contain concise tensors. Therefore, assume $n_1 \leq 2n_2$ and $n_2 \leq 2n_1$. 

\begin{itemize}
\item If $n_1 < n_2 \leq 2n_1$, there exist unique $p \geq 0$, and $\alpha,\beta \geq 0$ with $\beta > 0$ such that  
\[
n_1 = p \alpha + (p+1) \beta,  \qquad n_2 = (p+1) \alpha + (p+2) \beta 
 \]
 and a generic pencil $T$ of size $n_1 \times n_2$ has the form 
 \[
T = L_p^{\boxplus \alpha} \boxplus L_{p+1}^{\boxplus \beta}.
 \]
 \item If $n_2 < n_1 \leq 2n_2$, there exist unique $p \geq 0$, and $\alpha,\beta \geq 0$ with $\beta > 0$ such that  
\[
n_2 = p \alpha + (p+1) \beta,  \qquad n_1 = (p+1) \alpha + (p+2) \beta 
 \]
 and a generic pencil $T$ of size $n_1 \times n_2$ has the form 
 \[
T = R_p^{\boxplus \alpha} \boxplus R_{p+1}^{\boxplus \beta}.
 \]
 \item If $n_1 = n_2$, there exist $z_1 \vvirg z_{n_1}$ distinct linear forms such that a generic pencil $T$ of size $n_1 \times n_2$ has the form
 \[
T = \bigboxplus_{i=1}^{n_1} J_1(z_i).
 \]
\end{itemize}
In \cite[Theorem 3]{Pok86}, degeneration between pencils of matrices is characterized in terms of the sizes and the multiplicities of the indecomposable blocks. We will use some specific instances of this characterization in the following sections.

\subsection{Orbits, isotropy groups and Lie algebras}

Let $G$ be a complex algebraic group acting on a vector space $V$. In the cases of interest for us, we will have, mainly, $G = \GL(V_1) \ttimes \GL(V_d)$ acting on $V = V_1 \ootimes V_d$. Let $v \in V$ be a vector. The orbit of $v$ is the set $G \cdot v = \{ g \cdot v : g \in G\}$; the orbit-closure of $v$ is $\bar{G \cdot v}$, where the closure can be taken equivalently in the Zariski or the Euclidean topology. 

The isotropy group of $v$ is the stabilizer of $v$ under the action of $G$, that is 
\[
\Stab_G (v) = \{ g \in G : g \cdot v = v \}.
\]
The isotropy group is an algebraic subgroup of $G$. It is union of finitely many (irreducible) connected components. The identity component $\Stab_G^\circ(v)$ is the irreducible component containing the identity element of $G$, and it is a normal subgroup of $\Stab_G(v)$ \cite[Ch. 1, Sec. 1.2]{Borel}. There is a natural orbit map 
\begin{align*}
\gamma: G &\to V \\
g & \mapsto g \cdot v
\end{align*}
whose image is $G \cdot v$. In particular, the dimension of the orbit-closure $\bar{G \cdot v}$ equals the rank of the differential of $\gamma$ at any point; taking the differential at the identity element, we have 
\begin{align*}
\mathrm{d} \gamma: \frakg &\to V \\
X & \mapsto X. v
\end{align*}
where $\frakg$ is the Lie algebra of $G$ and we identify $V$ with the tangent space $T_vV$. The kernel of the differential is 
\[
\ker  \mathrm{d} \gamma = \ann_\frakg(v) = \{ X \in \frakg : X.v = 0 \},
\]
and it coincides with the Lie algebra of $\Stab_G(v)$.
In particular,
\[
\dim(G\cdot v)=\dim\mathfrak g-\dim\ann_{\mathfrak g}(v).
\]
In summary, the dimension of an orbit-closure is entirely controlled by the dimension of the annihilator $\ann_\frakg(v)$. 

We record some results for the action of $\GL(V_0) \times \GL(V_1) \times \GL(V_2)$ on $V_0 \otimes V_1 \otimes V_2$ which will be useful in the following. The action has a two dimensional kernel subgroup 
\[
(\bbC^\times)^2 = \{ (\lambda_0 \id_{V_0} , \lambda_1 \id_{V_1}, \lambda_2 \id_{V_2}) : \lambda_0 \lambda_1\lambda_2 = 1 \} \subseteq \GL(V_0) \times \GL(V_1) \times \GL(V_2);
\]
its Lie algebra is 
\[
\bbC^2 = \{ (z_0 \id_{V_0} , z_1 \id_{V_1}, z_2 \id_{V_2}) : z_0 + z_1 + z_2  = 0 \} \subseteq \frakgl(V_0) \times \frakgl(V_1) \times \frakgl(V_2).
\]
In particular, for every $T$, we have $\dim \Stab_G(T) = \dim \ann_\frakg(T) \geq 2$.

The following is an extension of \cite[Theorem 4.1]{CGLVW}, about annihilator of tensors with a direct sum structure.
\begin{lemma}\label{lemma: direct sum ann}
 Let $T_1 \in W \otimes V_1^{(1)} \otimes V_2^{(1)} \otimes V_3^{(1)}$ and $T_2 \in W \otimes V_1^{(2)} \otimes V_2^{(2)} \otimes V_3^{(2)}$ be two tensors and let $\frakg$ be a Lie algebra acting on $W$. Let $V_i = V_i^{(1)} \oplus V_i^{(2)}$ for $i=1,2,3$ and consider the tensor
 \[
T = T_1 + T_2 \in W \otimes V_1 \otimes V_2 \otimes V_3.
 \]
 Write $\tilde{\frakg} = \frakg \times \frakgl(V_1) \times \frakgl(V_2) \times \frakgl(V_3)$ and $\tilde{\frakg}_1,\tilde{\frakg}_2$ for the similar algebras with $\frakgl(V_i^{(1)}),\frakgl(V_i^{(2)})$. If $T_s$ is concise on the factors $V_i^{(s)}$, then  
\[
\ann_{\tilde{\frakg}} (T) \subseteq \ann_{\tilde{\frakg}_1} (T_1) + \ann_{\tilde{\frakg}_2} (T_2),
\]
where $\frakg_1 + \frakg_2$ is regarded as a subalgebra of $\frakg$.
\end{lemma}
\begin{proof}
The proof is similar to the one of \cite[Theorem 4.1]{CGLVW}. Let $(Y,X_1,X_2,X_3) \in \ann_{\tilde{\frakg}} (T)$. Then 
\[
0 = (Y,X_1,X_2,X_3).T = (Y,X_1,X_2,X_3).T_1 + (Y,X_1,X_2,X_3).T_2.
\]
Observe that the two tensors $(Y,X_1,X_2,X_3).T_1 , (Y,X_1,X_2,X_3).T_2$ belong to linear subspaces of $W \otimes V_1 \otimes V_2 \otimes V_3$ with trivial intersection. Indeed, since $T_s \in W \otimes V_1^{(s)} \otimes V_2^{(s)} \otimes V_3^{(s)}$, we have
\[
(Y,X_1,X_2,X_3).T_s \in W \otimes \bigoplus_{\substack{(s_1',s_2',s_3') : \\
s_i' = s \text{ for at least two }i}} V_1^{(s_1')} \otimes V_2^{(s_2')} \otimes V_3^{(s_3')}.
\]
The two spaces for $s = 1,2$ intersect trivially. This guarantees $0 = (Y,X_1,X_2,X_3).T_s$ for $s = 1,2$.

The conciseness assumption guarantees that if $(Y,X_1,X_2,X_3).T_s = 0$ then 
\[
(Y,X_1,X_2,X_3)|_{W \otimes V_1^{(s)} \otimes V_2^{(s)} \otimes V_3^{(s)}} \in \Hom ( W \otimes V_1^{(s)} \otimes V_2^{(s)} \otimes V_3^{(s)}, W \otimes V_1 \otimes V_2 \otimes V_3)
\]
is an element of $\ann_{\tilde{\frakg}_s}(T_s)$. This concludes the proof.
\end{proof}
The case of \cite[Theorem 4.1]{CGLVW} corresponds to the case where $W$ is a trivial one-dimensional tensor factor, so that $T_1,T_2$ are tensors of order three; \cite[Theorem 4.1]{CGLVW} shows that in this case the inclusion holds with equality. It is immediate that the statement holds with equality whenever the action of $\frakg$ is trivial. We will apply \autoref{lemma: direct sum ann} in the case where $W$ is itself a tensor space and $\frakg$ is a product of Lie algebras of general linear groups acting on its factors.

It is important in \autoref{lemma: direct sum ann} that $T$ has a direct sum structure on (at least) three tensor factors. If the direct sum is only on two tensor factors, then there are \emph{off-diagonal} terms that may contribute to the annihilator in a non-trivial way. Block sums of matrix pencils fall into this case. However, the off-diagonal contribution is easily characterized in the following result.

\begin{lemma}\label{lemma: ann for block sum pencils}
 Let $T_1 \in W \otimes V_1^{(1)} \otimes V_2^{(1)}$ and $T_2 \in W \otimes V_1^{(2)} \otimes V_2^{(2)}$ be two tensors and let $\frakg$ be a Lie algebra acting on $W$. Let $V_i = V_i^{(1)} \oplus V_i^{(2)}$ for $i=1,2$ and consider the tensor
 \[
T = T_1 + T_2 \in W \otimes V_1 \otimes V_2.
 \]
 Write $\tilde{\frakg} = \frakg \times \frakgl(V_1) \times \frakgl(V_2)$ and $\tilde{\frakg}_1,\tilde{\frakg}_2$ for the similar algebras with $\frakgl(V_i^{(1)}),\frakgl(V_i^{(2)})$. If $T_s$ is concise on the factors $V_i^{(s)}$, then  
\[
\ann_{\tilde{\frakg}} (T) \subseteq  \ann_{\tilde{\frakg}_1} (T_1) + \ann_{\tilde{\frakg}_2} (T_2) + M_1 + M_2,
\]
where $\frakg_1 + \frakg_2$ is regarded as a subalgebra of $\frakg$ and 
\begin{align*}
M_1 &= \{ (X_{21},Y_{12}) \in \Hom(V_1^{(1)},V_1^{(2)}) \oplus \Hom(V_2^{(2)},V_2^{(1)}) : X_{21}T_1 + Y_{12}T_2 = 0 \} \\
M_2 &= \{ (X_{12},Y_{21}) \in \Hom(V_1^{(2)},V_1^{(1)}) \oplus \Hom(V_2^{(1)},V_2^{(2)}) : X_{12}T_2 + Y_{21}T_1 = 0 \}. 
\end{align*}
\end{lemma}
\begin{proof}
The proof is similar to the one of \autoref{lemma: direct sum ann}. We only need to show that if $(Y,X_1,X_2)$ has a nonzero component in $\Hom(V_1^{(1)},V_1^{(2)}) \oplus \Hom(V_2^{(2)},V_2^{(1)})$, then such component belongs to $M_1$. The case for $M_2$ is similar.

Suppose $(Z,X,Y) \in \frakg \times \frakgl(V_1) \times \frakgl(V_2)$ be an element such that $(Z,X,Y).T = 0$ with $(X,Y)$ having a nonzero component in $\Hom(V_1^{(1)},V_1^{(2)}) \oplus \Hom(V_2^{(2)},V_2^{(1)})$. For every $Z \in \frakg$, we have $Z.T_s \in W \otimes V_1^{(s)} \otimes V_2^{(s)}$ so $Z.T$ has trivial components in the spaces $W \otimes V_1^{(1)} \otimes V_2^{(2)}$ and $W \otimes V_1^{(2)} \otimes V_2^{(1)}$. Therefore, if $(Z,X,Y).T = 0$, the component of $(X,Y)$ in $\Hom(V_1^{(1)},V_1^{(2)}) \oplus \Hom(V_2^{(2)},V_2^{(1)})$ already annihilates $T$. By the conciseness assumption, this is equivalent to the fact that such component belongs to $M_1$.  This concludes the proof.
\end{proof}

It is clear that the spaces $M_1,M_2$ in \autoref{lemma: ann for block sum pencils} are contained in the annihilator. Therefore, as in \autoref{lemma: direct sum ann}, if the action of $\frakg$ is trivial, we deduce that equality holds in the statement of \autoref{lemma: ann for block sum pencils}. 

We conclude this section with a result concerning the dimension of generic orbits in spaces of pencils; see also \cite[Lemma A.1]{GLS24}.
\begin{lemma}\label{prop: pencils orbits}
    Let $T \in \mathbb{C}^2 \otimes V_1 \otimes V_2$ be a generic tensor with $n_i = \dim V_i$.
    \begin{itemize}
    \item If $\dim V_1 \neq \dim V_2$, then the $(\GL(V_1) \times \GL(V_2))$-orbit of $T$ is dense in $\mathbb{C}^2 \otimes V_1 \otimes V_2$. In particular 
    \[
    \dim \ann_{\frakgl(V_1) \times \frakgl(V_2)}(T) = n_1^2 + n_2^2 - 2 n_1 n_2 = (n_1 - n_2)^2.
    \]
    \item If $n_1 = n_2$, then the $(\GL(V_1) \times \GL(V_2))$-orbit of $T$ has codimension $\dim V_1$. In particular
    \[
    \dim \ann_{\frakgl(V_1) \times \frakgl(V_2)}(T) = n_1^2 + n_2^2 - 2 n_1 n_2 + n_1 = n_1.
    \]
    \end{itemize}
\end{lemma}

\section{Triangle TNS for matrix pencils}\label{section: pencils}

In this section, we prove \autoref{main theorem} for the triangle network $\Delta$ with bond dimensions $\bfm = (m_{01},m_{12},m_{02})$ and physical dimensions $\bfn = (2,n_1,n_2)$ with $n_i = m_{0i}m_{12} - k_i$; see \autoref{fig: one 2}. 

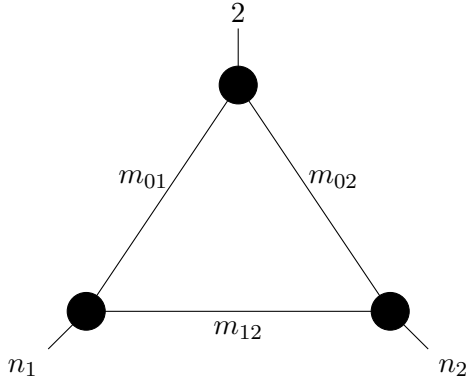
\begin{figure}[!h]
\centering
\begin{tikzpicture}[scale=.5]
\draw (0,0)--(4,6);
\draw (0,0)--(8,0);
\draw (8,0)--(4,6);
\node () at (1.5,3.5) {$m_{01}$};
\node () at (6.5,3.5) {$m_{02}$};
\node () at (4,-0.5) {$m_{12}$};
\draw (0,0)--(-1,-1);
\draw (4,6)--(4,7.5);
\draw (8,0)--(9,-1);
\draw[fill=black] (0,0) circle (.5cm);
\draw[fill=black] (4,6) circle (.5cm);
\draw[fill=black] (8,0) circle (.5cm);
\node[anchor = north east] at (-1,-1){$n_1$};
\node[anchor = north west] at (9,-1){$n_2$};
\node[anchor = south] at (4,7.4){$2$};
\node[anchor = north east] at (-1,-1){\phantom{$\bbC^{n_1}$}};
\node[anchor = north west] at (9,-1){\phantom{$\bbC^{n_2}$}};
\node[anchor = south] at (4,7.4){\phantom{$\bbC^{2}$}};
\node at (0,0){$1$};
\node at (4,6){$2$};
\node at (8,0){$0$};
\node at (-2,2){\phantom{$T = $}};
\end{tikzpicture}
\caption{Triangular tensor network $(\Delta, \bfm,\bfn)$ with bond dimensions $\bfm = (m_{01},m_{12},m_{02})$ and local dimensions $\bfn = (2,n_1,n_2)$.} \label{fig: one 2}
\end{figure}
A special role is played by the condition \eqref{eqn: nontriviality condition}, that we recall here
\begin{equation}
    \tag{\raisebox{0.3mm}{\textreferencemark}}
m_{01} = m_{02} \quad \text{ and } \quad  k_2 \leq k_1 < m_{12} .
\end{equation}
We aim to prove that if this condition does not hold, then $\calTNS^\Delta_{\bfm , \bfn}= V_0 \otimes V_1 \otimes V_2$, that is the tensor network variety fills the ambient space. We start proving this result if $m_{01} \neq m_{02}$.
\begin{theorem}\label{thm: different bond dimensions}
Let $\Delta$ be the triangle graph with bond dimensions $\bfm = (m_{01},m_{12},m_{02})$  with $m_{01} \neq m_{02}$ and local dimensions $\bfn = (2,n_1,n_2)$ with $n_i \leq m_{ij}m_{jk}$ for $\{i,j,k\} = \{0,1,2\}$.

Then
\[
\calTNS^\Delta_{\bfm , \bfn} = V_0 \otimes V_1 \otimes V_2 = \bbC^2 \otimes \bbC^{n_{1}} \otimes \bbC^{n_{2}}.
\]
\end{theorem}
\begin{proof}
    It suffices to prove the statement in the critical case, that is with $n_1 = m_{01}m_{12}$ and $n_2 = m_{02}m_{12}$. Indeed, the tensor network variety in the subcritical case coincides with a linear section of the one in the critical case. If the latter fills the ambient space, the former does as well.
    
    Under the assumption $m_{01} \neq m_{02}$, the space $V_0 \otimes V_1 \otimes V_2 = \bbC^2\otimes \bbC^{m_{12}m_{01}} \otimes \bbC^{m_{12}m_{02}}$ has a dense orbit for the action of $\GL(V_1) \times \GL(V_2)$, as explained in \autoref{sec: pencils}. To show that $\calTNS^\Delta_{\bfm,\bfn}$ fills the ambient space, it suffices to exhibit a linear map $X_0 \in \Hom(W_0,V_0)$ such that $X_0 \cdot T(\Delta,\bfm)$ is an element of such dense orbit.

    Since 
    \[
    \Hom(W_0,V_0) = V_0 \otimes E_{01} \otimes E_{02}^* =\bbC^2 \otimes \bbC^{m_{01}} \otimes \bbC^{m_{02}} ,
    \]
    there exists an element $X_0 \in \bbC^2 \otimes \bbC^{m_{01}} \otimes \bbC^{m_{02}} $ which has a dense orbit for the action of $\GL_{m_{01}} \times \GL_{m_{02}}$, because $m_{01} \neq m_{02}$. Assume without loss of generality that $m_{01} < m_{02}$; from the structure of the normal forms described in \autoref{sec: pencils}, the element having dense orbit can be chosen to be $X_0 = L_p^{\boxplus \alpha} \boxplus L_{p+1}^{\boxplus \beta}$, where $L_p$ is the $p$-th left singular pencil and $p,\alpha,\beta$ are the integers associated with the dimensions $(m_{01},m_{02})$ which are uniquely determined by the condition 
    \begin{equation}\label{eqn: dimensions in Hom}
    \begin{aligned}
    m_{01} &= \alpha \cdot p + \beta \cdot (p+1), \\
    m_{02} &= \alpha \cdot (p+1) + \beta \cdot (p+2).
    \end{aligned}
    \end{equation}
    Regard such $X_0$ as an element of $\Hom(W_0, V_0)$. By construction, the contraction against the graph tensor is 
    \[
    X_0 \cdot T(\Delta,\bfm) = (L_p^{\boxplus \alpha} \boxplus L_{p+1}^{\boxplus \beta}) \boxplus (\bfI_{m_{12}} \otimes e_0) = (L_p^{\boxplus m_{12}\alpha} \boxplus L_{p+1}^{\boxplus m_{12}\beta}).
    \]
    This is an element of the dense orbit in $\bbC^2 \otimes \bbC^{m_{01}m_{12}} \otimes \bbC^{m_{02}m_{12}}$ because multiplying the equations of \eqref{eqn: dimensions in Hom} by $m_{12}$ we see that $p,\alpha m_{12}, \beta m_{12}$ are the integers associated to the dimensions $(m_{01} m_{12} , m_{02}m_{12})$. Since this element belongs to $\calTNS^\Delta_{\bfm , \bfn}$, the dense orbit is contained in the tensor network variety, which therefore fills the ambient space.
\end{proof}

In the following, set $m = m_{01} = m_{02}$ so that $\bfm = (m,m_{12},m)$ and $\bfn = (2,n_1,n_2)$. We aim to establish the structure of the normal forms and compute the dimension of tensor network varieties $\calTNS^{\Delta}_{\bfm,\bfn}$ in this setting: this is done in two steps, first for the critical case, then for the subcritical case.

\subsection{Dimension of triangle TNS with \texorpdfstring{$(m,2,m)$}{(m,2,m)} corner: critical case}

Consider the tensor network as in \autoref{fig:mm1m:critical}. 
\begin{figure}[!h]
\centering
\begin{tikzpicture}[scale=.5]
\draw (0,0)--(4,6);
\draw (0,0)--(8,0);
\draw (8,0)--(4,6);
\node () at (1.5,3.5) {$m$};
\node () at (6.5,3.5) {$m$};
\node () at (4,-0.5) {$m_{12}$};
\draw (0,0)--(-1,-1);
\draw (4,6)--(4,7.5);
\draw (8,0)--(9,-1);
\draw[fill=black] (0,0) circle (.5cm);
\draw[fill=black] (4,6) circle (.5cm);
\draw[fill=black] (8,0) circle (.5cm);
\node[anchor = north east] at (-1,-1){$m_{12}m$};
\node[anchor = north west] at (9,-1){$m_{12}m$};
\node[anchor = south] at (4,7.4){$2$};
\node[anchor = north east] at (-1,-1){\phantom{$\bbC^{n_1}$}};
\node[anchor = north west] at (9,-1){\phantom{$\bbC^{n_2}$}};
\node[anchor = south] at (4,7.4){\phantom{$\bbC^{n_0}$}};
\node at (0,0){$1$};
\node at (4,6){$0$};
\node at (8,0){$2$};
\node at (-2,2){\phantom{$T = $}};
\end{tikzpicture}
\caption{Triangular tensor network with bond dimensions $\mathbf{m} = (m,m_{12},m)$ and $\mathbf{n} = (2,m_{12}m,m_{12}m)$.}\label{fig:mm1m:critical}
\end{figure}
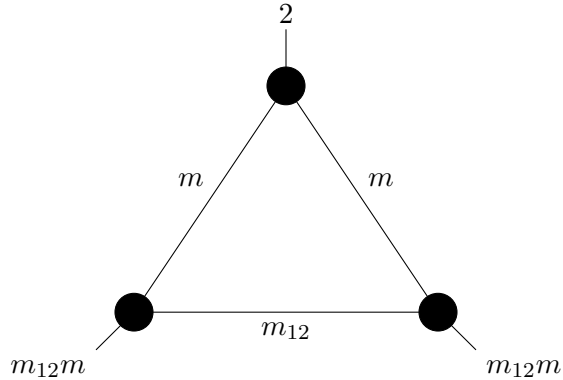

We are going to determine the normal form of an element of $\calTNS^\Delta_{\bfm , \bfn}$ and compute the dimension of the tensor network variety. This will establish \autoref{main theorem}, \autoref{Part1} and \autoref{Part2}, in the critical case.

\begin{theorem}[Dimension with two  equal critical vertices]\label{thm: triangle mm1}
Let $\Delta$ be the triangle graph with bond dimensions $\bfm = (m,m_{12},m)$, and local dimensions $\bfn = (2, m_{12}m, m_{12}m)$. Let $T$ be a generic element of $\calTNS^\Delta_{\bfm , \bfn}$. Then, after a suitable change of coordinates, regarded as a matrix of elements in $V_0 \simeq \bbC^2$,
\[
T = \bigboxplus_{i=1}^m (J_1(z_i) \boxtimes \bfI_{m_{12}} )
\]
for generic $z_i \in V_0$; moreover
\[
\dim \calTNS^\Delta_{\bfm , \bfn} = (2m-1)m m_{12}^2 + m .
\]
\end{theorem}

\begin{proof}
    Let $T \in \calTNS^\Delta_{\bfm,\bfn}$ be a generic element. We are going to show that up to changing coordinates $T = \bigboxplus_{i=1}^m (J_1(z_i) \boxtimes \bfI_{m_{12}} )$. Write
    \[
    T = (X_0 \otimes X_1 \otimes X_2)T(\Delta,\bfm)
    \]
    where $T(\Delta,\bfm) $ is the graph tensor with bond dimensions $(m,m_{12},m)$. 
    
    For $i=1,2$, we have $\dim V_i = \dim W_i = m_{12}m$, so, after suitable identification, the genericity condition allows us to normalize $X_1 = X_2 = \id_{\bbC^{mm_{12}}}$. Therefore $T$ is uniquely determined by $X_0 \in \Hom(W_0 , V_0) = V_0 \otimes E_1 \otimes E_2^* = \bbC^2 \otimes {\bbC^m}^* \otimes (\bbC^m)^*$. 

    By the genericity of $X_0$, we have 
    \[
    X_0 = \bigboxplus_{i=1}^m J_1(z_i)
    \]
    for certain elements $z_i = v_0 + \zeta_i v_1  \in V_0 = \bbC^2$; here $v_0,v_1$ is a fixed basis of $V_0$. We obtain
    \[
    T = (X_0 \otimes \id_{W_1} \otimes \id_{W_2}) T(\Delta,\bfm) = \bfI_{m_{12}} \boxtimes X_0 = \bfI_{m_{12}} \boxtimes \left(\bigboxplus_{i=1}^m J_1(z_i)\right).
    \]
    This shows the statement about the normal form and proves that $\calTNS^\Delta_{\bfm,\bfn}$ is the closure of the set of pencils of matrices whose Kronecker normal form is diagonal with $m$ diagonal blocks $z_i \cdot \bfI_{m _{12}}$.

    It remains to compute the dimension of the tensor network variety. By the characterization above, an open subset of $\calTNS^{\Delta}_{\bfm,\bfn}$ is parametrized by a map
    \begin{align*}
    \tilde{\Phi} : \bbC^m \times \GL(V_1) \times \GL(V_2)  &\to V_0 \otimes V_1 \otimes V_2 \\
    ((\zeta_1 \vvirg \zeta_m) , g_1,g_2) &\mapsto (g_1 \otimes g_2) ( \bigboxplus_{i =1}^m J_1( v_0 + \zeta_i v_1) \boxtimes \bfI_{m_{12}}).
    \end{align*}
    Since the action of $\GL(V_1) \times \GL(V_2)$ does not change the generalized eigenvalues of a pencil, other than rescaling them, the generic fiber of $\tilde{\Phi}$  is the stabilizer 
    \[
    \Stab_{\GL(V_1) \times \GL(V_2)} ( \bigboxplus_{i =1}^m J_1( v_0 + \zeta_i v_1) \boxtimes \bfI_{m_{12}})
    \]
    for a generic choice of $(\zeta_1 \vvirg \zeta_m) \in \bbC^m$. We deduce 
    \begin{align*}
   \dim \calTNS^{\Delta}_{\bfm,\bfn} &=  2(mm_{12})^2 + m - \dim \Stab_{\GL(V_1) \times \GL(V_2)} ( \bigboxplus_{i =1}^m J_1( v_0 + \xi_i v_1)\boxtimes \bfI_{m_{12}}) \\
   &= 2(mm_{12})^2 + m - \dim \ann_{\frakgl(V_1) \oplus \frakgl(V_2)} (\bigboxplus_{i =1}^m J_1( v_0 + \xi_i v_1)\boxtimes \bfI_{m_{12}}).
    \end{align*}
This is computed applying iteratively \autoref{lemma: ann for block sum pencils} on the tensors $J_1( v_0 + \xi_i v_1)\boxtimes \bfI_{m_{12}}$. The spaces $M_1,M_2$ in the statement of \autoref{lemma: ann for block sum pencils} are trivial because the generalized eigenvalues of the different Jordan blocks are distinct. Therefore, the annihilator
\[
\ann_{\frakgl(V_1) \oplus \frakgl(V_2)} (\bigboxplus_{i =1}^m J_1( v_0 + \xi_i v_1)\boxtimes \bfI_{m_{12}})
\]
is simply the sum of the annihilator of its block-summands, acting on the corresponding subspaces. We deduce the following:
\[
\ann_{\frakgl(V_1) \oplus \frakgl(V_2)}(\bigboxplus_{i =1}^m J_1( v_0 + \xi_i v_1)\boxtimes \bfI_{m_{12}}) = \bigoplus_{1}^m \ann_{\frakgl_{m_{12}} \times \frakgl_{m_{12}}} (v_0+\xi_i v_1) \bfI_{m_{12}} \simeq \bigoplus_{1}^m \frakgl_{m_{12}}.
\]
We conclude that
\[
\dim \calTNS^{\Delta}_{\bfm,\bfn} = 2(mm_{12})^2 + m - m m_{12}^2 = (2m-1)mm_{12}^2 + m.
\]
\end{proof}

\begin{corollary}[Defectivity with two equal critical vertices]\label{cor:triangle_defect}
In the setting of \autoref{thm: triangle mm1}, the fiber defect and the defect are 
\[
\mathrm{fiberdefect}^\Delta_{\mathbf{m},\mathbf{n}}
= \delta^\Delta_{\bfm,\bfn}= (m-1)(m_{12}^2-1).
\]
\end{corollary}
\begin{proof}
Evaluating the expected dimension in \eqref{eq:expdim} with $m_{01}=m_{02}=m$ and $\bfn=(2,mm_{12},mm_{12})$, we observe that the minimum is attained by the first term and its value is
\[
\expdim(\calTNS^{\Delta}_{\bfm,\bfn})= 2m^2m_{12}^2-m_{12}^2+1.
\]
Using the dimension formula computed in \autoref{thm: triangle mm1}, we conclude
\[
\mathrm{fiberdefect}^\Delta_{\mathbf{m},\mathbf{n}} = \delta^\Delta_{\bfm,\bfn} =\expdim \calTNS^{\Delta}_{\bfm,\bfn}-\dim \calTNS^\Delta_{\bfm,\bfn} =(m-1)(m_{12}^2-1).
\]
\end{proof}

\subsection{Dimension of triangle TNS with \texorpdfstring{$(m,2,m)$}{(m,2,m)} corner: subcritical case}

Consider the tensor network as in \autoref{fig:mm1m:subcritical}. 
\begin{figure}[!h]
\centering
\begin{tikzpicture}[scale=.5]
\draw (0,0)--(4,6);
\draw (0,0)--(8,0);
\draw (8,0)--(4,6);
\node () at (1.5,3.5) {$m$};
\node () at (6.5,3.5) {$m$};
\node () at (4,-0.5) {$m_{12}$};
\draw (0,0)--(-1,-1);
\draw (4,6)--(4,7.5);
\draw (8,0)--(9,-1);
\draw[fill=black] (0,0) circle (.5cm);
\draw[fill=black] (4,6) circle (.5cm);
\draw[fill=black] (8,0) circle (.5cm);
\node[anchor = north east] at (-1,-1){$m_{12}m - k_1$};
\node[anchor = north west] at (9,-1){$m_{12}m - k_2$};
\node[anchor = south] at (4,7.4){$2$};
\node[anchor = north east] at (-1,-1){\phantom{$\bbC^{n_1}$}};
\node[anchor = north west] at (9,-1){\phantom{$\bbC^{n_2}$}};
\node[anchor = south] at (4,7.4){\phantom{$\bbC^{n_0}$}};
\node at (0,0){$1$};
\node at (4,6){$0$};
\node at (8,0){$2$};
\node at (-2,2){\phantom{$T = $}};
\end{tikzpicture}
\caption{Triangular tensor network with bond dimensions $\mathbf{m} = (m,m_{12},m)$ and $\mathbf{n} = (2,m_{12}m - k_1,m_{12}m-k_2)$.}\label{fig:mm1m:subcritical}
\end{figure}
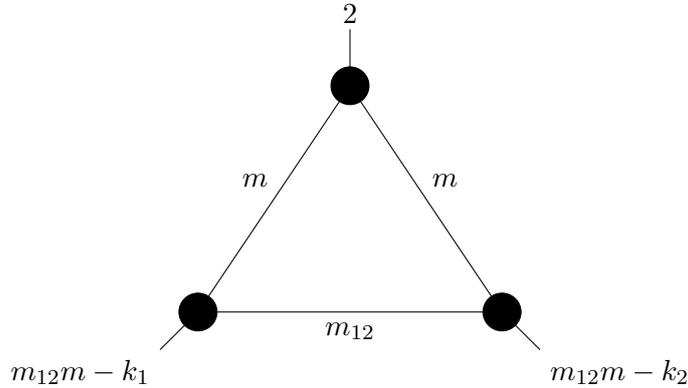

In this case, we will determine the normal form of elements of the tensor network variety and its dimension, completing the proof of \autoref{main theorem}, \autoref{Part1} and \autoref{Part2}.

\begin{theorem}\label{thm: triangular subcritical}
Let $\Delta$ be the triangle graph with bond dimensions $\bfm = (m,m_{12},m)$, and local dimensions $\bfn = (2,m_{12}m-k_1, m_{12}m-k_2)$ with $k_2 \leq k_1$. Let $T$ be a generic element of $\calTNS^\Delta_{\bfm , \bfn}$. Then, after a suitable change of coordinates, $T = T_1 \boxplus T_2$ where, regarded as a matrix of elements in $V_0 \simeq \bbC^2$,
\[
T_1 = \bigboxplus_{i=1}^{m} (J_1(z_i) \boxtimes \bfI_{m_{12}-k_1} )
\]
for generic $z_i \in V_0$ and 
\[
T_2 \in V_0 \otimes \bbC^{(m-1)k_{1}} \otimes \bbC^{(m-1) k_{1} + (k_{1}-k_{2})}
\]
is generic. Moreover
\[
\dim \calTNS^\Delta_{\bfm , \bfn} =  2 m^2 m_{12}^{2} - m m_{12}^{2} - m m_{12} k_{1} - m m_{12} k_{2} - m k_{1} k_{2} + 2 k_{1} k_{2} + m .
\]
\end{theorem}
\begin{proof} The proof is divided into several steps: after setting up the notation, we compute the Kronecker invariants of the pencils in the relevant tensor network varieties; we use the characterization of \cite{Pok86} to show that the pencils described in the statement belong to the tensor network variety and that any generic element of the tensor network variety is a limit of such pencils; finally, we compute the dimension of the variety using \autoref{lemma: ann for block sum pencils}.

\medskip
{\textbf{Setup and notation.}} We use the characterization of \cite{Pok86} of degeneration of matrix pencils and the characterization of the generic elements of $\calTNS^\Delta_{\bfm, (m_{12}m, m_{12}m, 2)}$ obtained in \autoref{thm: triangle mm1}. It is useful to fix the following notation: 
\begin{align*}
A(\zeta_1 \vvirg \zeta_m) &= \bigboxplus_{i=1}^{m} J_1(v_0 + \zeta_i v_1) \boxtimes \bfI_{m_{12}} , \\
B(\zeta_1 \vvirg \zeta_m) &= \bigboxplus_{i=1}^{m} J_1(v_0 + \zeta_i v_1) \boxtimes \bfI_{m_{12}-k_1} \boxplus T_2,
\end{align*} 
for some (generic) $\zeta_i \in \bbC$ and $T_2 \in V_0 \otimes \bbC^{(m-1)k_{1}} \otimes \bbC^{(m-1) k_{1} + (k_{1}-k_{2})}$ generic. By \autoref{thm: triangle mm1}, $A(\zeta_1 \vvirg \zeta_m)$ is a generic element of $\calTNS^\Delta_{\bfm, (m_{12}m, m_{12}m, 2)}$. Moreover $B(\zeta_1 \vvirg \zeta_m)$ is an element of $V_0 \otimes V_1 \otimes V_2$ described in the statement of the Theorem.

We are going to prove that $B(\zeta_1 \vvirg \zeta_m)$ is a degeneration of $A(\zeta_1 \vvirg \zeta_m)$. This shows that $B(\zeta_1 \vvirg \zeta_m) \in \calTNS^\Delta_{\bfm , \bfn}$. Moreover, we show that if $B' \in V_0 \otimes V_1 \otimes V_2$ is a degeneration of $A(\zeta_1 \vvirg \zeta_m)$ then $B'$ is limit of a sequence of elements of the form $B(\zeta_1 \vvirg \zeta_m)$. This will provide the statement about the normal form. The computation of the dimension will then follow using \autoref{lemma: ann for block sum pencils}.

\medskip
{\textbf{Step 1: Computation of Kronecker invariants.}}
To show that $B(\zeta_1 \vvirg \zeta_m)$ is a degeneration of $A(\zeta_1 \vvirg \zeta_m)$ we use \cite[Thm. 3]{Pok86} which relies on certain inequalities among the Kronecker invariants of the pencils. Regard $B(\zeta_1 \vvirg \zeta_m)$ as an element of $V_0 \otimes W_1 \otimes W_2$ under suitable embeddings $V_1 \to W_1$ and $V_2 \to W_2$. 

For any pencil $P \in V_0 \otimes W_1 \otimes W_2$, let $d_{p,\zeta}(P)$ be the number of Jordan blocks of size $p$ with generalized eigenvalue $v_0 + \zeta v_1$ in the Kronecker decomposition of $P$. We have 
\begin{equation}\label{eqn: d invariants Pok}
\begin{aligned}
d_{p,\zeta}(A(\zeta_1 \vvirg \zeta_m)) &= \left\{\begin{array}{ll}
m_{12} & \text{if $p =1$ and $\zeta=\zeta_i$ for some $i$,}\\
0 & \text{otherwise},
\end{array}
\right. \\
d_{p,\xi}(B(\zeta_1 \vvirg \zeta_m)) &= \left\{\begin{array}{ll}
m_{12}-k_1 & \text{if $p =1$ and $\xi=\xi_i$ for some $i$,}\\
0 & \text{otherwise}.
\end{array}
\right. 
\end{aligned}
\end{equation}
Similarly, for a pencil $P$, let $\ell_p(P)$ (resp. $r_p(P)$) be the number of left (resp. right) singular blocks of size $p$ in the Kronecker decomposition of $P$. Kronecker blocks of size $0\times 1$ and $1 \times 0$ reflect the fact that $P$ is not concise as a tensor in $V_0 \otimes W_1 \otimes W_2$. Let $\ell(P)$ (resp. $r(P)$) be the total number of left (resp. right) singular blocks in the Kronecker decomposition of $P$. We have 
\begin{align*}
    r_p(A(\zeta_1 \vvirg \zeta_m)) &= \ell_p(A(\zeta_1 \vvirg \zeta_m)) = 0 \text{ for all } p, \\
r_0(B(\zeta_1 \vvirg \zeta_m)) &= r(B(\zeta_1 \vvirg \zeta_m)) = k_{1},\\
\ell_0(B(\zeta_1 \vvirg \zeta_m)) &= k_{2} \\
\ell_{\bar{p}}(B(\zeta_1 \vvirg \zeta_m)) &\neq 0 \text{ and } \ell_{\bar{p}+1}(B(\zeta_1 \vvirg \zeta_m)) \neq 0.
\end{align*}
The integer $\bar{p}$ is determined by the Kronecker structure of the generic pencil in $\bbC^{(m-1)k_{1}} \otimes \bbC^{(m-1) k_{1} + (k_{1}-k_{2})} \otimes V_3$ as explained in \autoref{sec: pencils}. Such a condition in particular implies that 
\[
\ell_{\bar{p}}(B(\zeta_1 \vvirg \zeta_m)) + \ell_{\bar{p}+1}(B(\zeta_1 \vvirg \zeta_m)) = k_{1} - k_{2}.
\] 
Now, \cite[Thm. 3]{Pok86} states that $B(\zeta_1 \vvirg \zeta_m)$ is a degeneration of $A(\zeta_1 \vvirg \zeta_m)$ if and only if three inequalities between these structural coefficients are satisfied. Since $\ell_p(A(\zeta_1 \vvirg \zeta_m)) = r_p(A(\zeta_1 \vvirg \zeta_m)) = 0$, Equations (8) and (9) in \cite[Thm. 3]{Pok86} are immediately satisfied. Equation (10)  for the degeneration of $A(\zeta_1 \vvirg \zeta_m)$ into $B(\zeta_1 \vvirg \zeta_m)$ reduces to
\[
d_{1,\zeta_i}(A(\zeta_1 \vvirg \zeta_m)) \leq j \cdot k_{1} + d_{1,\zeta_i}(B(\zeta_1 \vvirg \zeta_m)) \quad \text{ for all } j \geq 1
\]
which is satisfied because of \autoref{eqn: d invariants Pok} since $j \geq 1$.

This shows that $B(\zeta_1 \vvirg \zeta_m)$ is a degeneration of $A(\zeta_1 \vvirg \zeta_m)$ and therefore all tensors described in the statement of the theorem belong to $\calTNS^\Delta_{\bfm,\bfn}$.

\medskip
{\textbf{Step 2: Generic form of elements in $\calTNS^\Delta_{\bfm,\bfn}$.}}
To show that the tensors of the form $B(\zeta_1 \vvirg \zeta_m)$ form a dense subset of $\calTNS^\Delta_{\bfm,\bfn}$, we are going to show that if $C$ is a generic enough element of $\calTNS^{\Delta}_{\bfm,\bfn}$ then there exists a choice of $\zeta_1 \vvirg \zeta_m$ and $T_2$ such that 
\[
B(\zeta_1 \vvirg \zeta_m) = \bigboxplus_{i=1}^{m} (J_1(v_0 + \zeta_i v_1) \boxtimes \bfI_{m_{12}-k_1} ) \boxplus T_2
\]
degenerates to $C$. If $C \in \calTNS^\Delta_{\bfm,\bfn}$ is generic enough, then $C$ is a degeneration of $A(\zeta_1 \vvirg \zeta_m) \in \calTNS^\Delta_{\bfm, (m_{12}m, m_{12}m, 2)}$ for some choice of $(\zeta_1 \vvirg \zeta_m) \in \bbC^m$. This degeneration can be interpreted in terms of the inequalities of \cite[Thm. 3]{Pok86}. If $C$ is generic enough, then $C$ is concise in $\bbC^2 \otimes \bbC^{n_1} \otimes \bbC^{n_2}$, so that, regarding it as an element of $V_0 \otimes W_1 \otimes W_2 $, we have 
\[
r_0(C) = k_1, \qquad \ell_0(C) = k_2.
\]
Since $A(\zeta_1 \vvirg \zeta_m)$, regarded as a matrix of linear forms, is generally full rank, we have that $C$ is generically of rank $n_1$ because $n_1 \leq n_2$. In particular, $C$ has no other right singular blocks besides the $k_1$ blocks of size $0 \times 1$. Therefore $r_p(C) = 0$ for all $p \geq 1$ and $r(C) = r_0(C) = k_1$. This shows that the right singular component of $C$ coincides with the right singular component of $B(\zeta_1 \vvirg \zeta_m)$, for any choice of $T_2 \in \bbC^2 \otimes  \bbC^{(m-1)k_{1}} \otimes \bbC^{(m-1) k_{1} + (k_{1}-k_{2})}$.

Consider equation (10) of \cite[Thm. 3]{Pok86} applied to the degeneration of $A(\zeta_1 \vvirg \zeta_m)$ into $C$. In the instance $j = 1$, and for every generalized eigenvalue $z_i = v_0 + \zeta_i v_1$, we have 
\[
m_{12} \leq k_1 + d_{1,\zeta_i}(C)
\]
showing $d_{1,\zeta_i}(C) \geq m_{12} - k_1$ for every $i = 1 \vvirg m$. This shows that the pencil $C$ has at least $m$ distinct eigenvalues, each with multiplicity at least $m_{12} - k_1$. This shows $C = \bigboxplus_{i=1}^{m} J_1(v_0 + \zeta_i v_1) \boxplus C_2$, with $C_2 \in \bbC^2 \otimes  \bbC^{(m-1)k_{1}} \otimes \bbC^{(m-1) k_{1} + (k_{1}-k_{2})}$. Therefore, $C_2$ is degeneration of some suitable $T_2$ generic in $\bbC^2 \otimes  \bbC^{(m-1)k_{1}} \otimes \bbC^{(m-1) k_{1} + (k_{1}-k_{2})}$ and in turn we obtain that $C$ is a degeneration of $B(\zeta_1 \vvirg \zeta_m)$.

This concludes the proof concerning the normal form of tensors in $\calTNS^\Delta_{\bfm,\bfn}$.

\medskip
{\textbf{Step 3: Dimension computation.}}
The dimension of $\calTNS^\Delta_{\bfm,\bfn}$ is computed using \autoref{lemma: ann for block sum pencils}. Note that a generic element of $\calTNS^\Delta_{\bfm,\bfn}$ has $m$ distinct eigenvalues if $k_{1} \neq k_{2}$ and $m+ k_{1}(m-1)$ distinct eigenvalues if $k_{1} = k_{2}$. Therefore, we have 
\begin{equation}
\label{eqn:dimTNS with big GL}
\dim \calTNS^\Delta_{\bfm,\bfn} = \tilde{m} + \dim (\GL(V_{1}) \times \GL(V_{2}) \cdot T) 
\end{equation}
where 
\[
\tilde{m}:=\left\{ \begin{array}{ll}
    m,  & k_{1} \neq k_{2}, \\
    m+ k_{1}(m-1), & k_{1} = k_{2}
\end{array} \right.
\]
and $T$ is a generic element of $\calTNS^\Delta_{\bfm,\bfn}$.

Fix a generic $T  \in \calTNS^\Delta_{\bfm,\bfn}$ and write $T =  T_1 \boxplus T_2$ as in the characterization proved above; therefore 
\begin{align*}
T_1 &= \bigboxplus_{i=1}^{m} J_1(v_0 + \zeta_i v_1) \boxtimes \bfI_{m_{12}-k_{1}}  \in V_0 \otimes V_{1}' \otimes V_{2}' \\
T_2 &\text{ generic in } \in V_0 \otimes V_{1}'' \otimes V_{2}''
\end{align*}
with $\dim V_{1}' = \dim V_{2}' = m (m_{12}-k_{1})$ and $\dim V_{1}'' = (m-1)k_{1}$ and $\dim V_{2}'' = (m-1)k_{1} + (k_{1}-k_{2})$. We are going to apply \autoref{lemma: ann for block sum pencils}. With the same calculation as in the proof of \autoref{thm: triangle mm1}, we obtain
\[
\dim \ann(T_1) = m (m_{12}-k_{1})^2.
\]
Since $T_2$ is generic, using \autoref{prop: pencils orbits}, we obtain 
\begin{align*}
\dim \ann(T_2) &= ((m-1)k_{1})^2 +((m-1)k_{1} + (k_{1}-k_{2}))^2 \\
&- (2((m-1)k_{1})((m-1)k_{1} + (k_{1}-k_{2})) - c) = (k_{1}-k_{2})^2 + c,
\end{align*}
where $c$ is the codimension of the generic orbit in $\bbC^2 \otimes V_{1}'' \otimes V_{2}''$, that is 
\[
c: = \left\{\begin{array}{ll}
    0, &   \hbox{if } k_{1} \neq k_{2}\\
    (m-1)k_{1}, & \hbox{if } k_{1} = k_{2}.
\end{array}
\right.
\]
It remains to compute the dimension of the spaces $M_1,M_2$ of \autoref{lemma: ann for block sum pencils}. We will show $M_2 = 0$ and $\dim M_1 = m(m_{12}-k_{1})(k_{1}-k_{2})$. Note that $M_1,M_2$ are the kernels of the two linear maps:
\begin{align*}
\phi_1 : \Hom(V'_{1},V''_{1}) \oplus \Hom(V''_{2},V'_{2}) &\to \bbC^2 \otimes V''_{1} \otimes V'_{2} \\
(X,Y) &\mapsto X\cdot T_1 + Y \cdot T_2 \\
\phi_2 : \Hom(V''_{1},V'_{1}) \oplus \Hom(V'_{2},V''_{2}) &\to \bbC^2 \otimes V'_{1} \otimes V''_{2} \\
(X,Y) &\mapsto X\cdot T_2 + Y \cdot T_1.
\end{align*}

\medskip
{\textbf{Analysis of the maps $\phi_1$ and $\phi_2$.}}
To show that $\phi_2$ is injective, regard $T_1,T_2$ as matrices of linear forms, parametrizing the image of the flattening maps
\[
T_1 : {V_{2}'}^* \to \bbC^2 \otimes V_{1}' \qquad T_2 : {V_{2}''}^* \to \bbC^2 \otimes V_{1}''.
\]
Explicitly, we have 
\[
T_1 = \left[\begin{array}{ccc|c|ccc}
e_1 & \cdots & e_{m_{12}-k_{1}}& \cdots & e_{(m-1)(m_{12}-k_{1})+1} & \cdots & e_{m(m_{12}-k_{1})} \\
\zeta_1 e_1 & \cdots & \zeta_1e_{m_{12}-k_{1}} & \cdots &\zeta_m e_{(m-1)(m_{12}-k_{1})+1} & \cdots & \zeta_m e_{m(m_{12}-k_{1})} \\
\end{array} \right]
\]
where $e_{1} \vvirg e_{m(m_{12}-k_{1})}$ is a basis of $V_{2}'$. Similarly 
\[
T_2 = \left[ \begin{array}{ccc}
z_{1,1} & \cdots & z_{1,(m-1)k_{1}} \\
z_{2,1} & \cdots & z_{2,(m-1)k_{1}} \\
\end{array}
\right]
\]
where $z_{i,j}$ are generic elements of $V_{2}''$. The element $X \in \Hom(V_{1}'',V_{1}')$ acts on (the matrix of) $T_2$ by multiplication on the right. The element $Y \in \Hom(V_{2}',V_{2}'')$ acts on (the entries of) $T_1$ by linear change of coordinates. If $(X,Y) \in M_2$, then $X\cdot T_2 + Y \cdot T_1 = 0$, so $X \cdot T_2  = T_2 X^\bft$ has the same structure as $T_1$: it is a $2 \times m(m_{12}-k_{1})$ matrix consisting of $m$ submatrices of size $2 \times (m_{12}-k_{1})$ placed side by side, each of rank one. Since the entries of $T_2$ are generic and $\dim V_{1}'' \leq \dim V_{2}''$, this yields a contradiction. We obtain $M_2 = 0$. 

The same argument shows that if $k_{1} = k_{2}$ then $\phi_1$ is injective as well, because in this case $T_2$ is a square pencil. In particular $M_1 = 0$ if $k_{1} = k_{2}$. Moreover, in this case $\phi_1$ is surjective as well, because if $k_{1} = k_{2}$ the domain and the codomain of $\phi_1$ have the same dimension.

If $k_{1} > k_{2}$ then the map $\phi_1$ is even more evidently surjective, because it is surjective even if the linear map $Y$ is evaluated on a generic tensor $\tilde{T}_2 \in \bbC^2 \otimes V_{1}'' \otimes \bar{V_{2}''}$ for a subspace $\bar{V_{2}''} \subseteq V_{2}''$ of dimension equal to $\dim V_{1}''$. By upper semicontinuity of the matrix rank, $\phi_1$ is surjective. We obtain 
\begin{align*}
\dim M_1 = &\dim \ker \phi_1 = \\
&\dim \Hom ( V_{1}',V_{1}'') + \dim \Hom ( V_{2}'',V_{2}') - \dim ( \bbC^2 \otimes V_{1}'' \otimes V_{2}') = \\
&m(m_{12} -k_{1}) \cdot ((m-1)k_{1}) + ((m-1)k_{1} + (k_{1}-k_{2}) ) \cdot m(m_{12}-k_{1}) \\
& \qquad - 2 \cdot ((m-1)k_{1}) \cdot m(m_{12} -k_{1}) = \\ 
&m(m_{12}-k_{1})(k_{1} - k_{2}).
\end{align*}

\medskip
{\textbf{Step 4: Final formula.}}
Applying \eqref{eqn:dimTNS with big GL} and \autoref{lemma: ann for block sum pencils}, and using $\tilde{m} -c = m$, the dimension of $\mathcal{TNS}^\Delta_{\bfm,\bfn}$ turns out to be:
\begin{align*}
    \dim \mathcal{TNS}^\Delta_{\bfm,\bfn} = &\tilde{m} + \dim \GL(V_{1}) + \dim \GL (V_{2}) - \dim \ann(T)  = \\
    & \tilde{m} + n_{1}^2 + n_{2}^2 - \dim \ann(T_1) - \dim \ann(T_2) -  \dim M_1 - \dim M_2 = \\
    &\tilde{m} + n_{1}^2 + n_{2}^2 - [m (m_{12}-k_{1})^2] - [(k_{1}-k_{2})^2 + c] - \\
    &\qquad [m(m_{12}-k_{1})(k_{1}-k_{2})] - 0 = \\
    &m +  [m(m_{12}-k_{1}) + (m-1)k_{1})]^2 + \\
    & \qquad [m(m_{12}-k_{1}) + (m-1)k_{1} + (k_{1}-k_{2})]^2  -\\
    & \qquad m (m_{12}-k_{1})^2 - (k_{1}-k_{2})^2  - m(m_{12}-k_{1})(k_{1}-k_{2})\\
    =& 2m^2m_{12}^{2}-mm_{12}^{2}- mm_{12}k_{1} - mm_{12}k_{2} - mk_{1}k_{2} +2k_{1}k_{2}+m.
\end{align*}
This concludes the proof concerning the dimension of $\calTNS^\Delta_{\bfm,\bfn}$.
\end{proof}

\begin{corollary}\label{cor:defect_subcritical}
    In the setting of \autoref{thm: triangular subcritical}, the fiber defect of $\calTNS^\Delta_{\bfm,\bfn}$ is 
\[
\mathrm{fiberdefect}^\Delta_{\bfm,\bfn} = 
(m-2)k_1k_2+(m-1)(m_{12}^2-1).
\]
\end{corollary}
\begin{proof}
Subtract the dimension computed in \autoref{thm: triangular subcritical} from the first term in the formula for the expected dimension in \eqref{eq:expdim}.
\end{proof}

\section{Equations for tensor network triangles and coincident root loci}\label{sec:4}

In this section, we study the equations of the varieties $\calTNS^\Delta_{\bfm,\bfn}$ in the case $\bfn = (2,n_1,n_2)$ and we prove \autoref{main theorem}, \autoref{Part3}. 

Let $\dim V_0 = 2$ and let $S^d V_0$ be the space of homogeneous polynomials of degree $d$ in two variables $v_0,v_1$. Let $\lambda = (\lambda_1 \vvirg \lambda_s)$ be a partition of $d$, that is a sequence of nonnegative integers with $\lambda_i \geq \lambda_{i+1}$ and $\sum_{i=1}^s \lambda_i = d$. The \emph{coincident root locus} $\calC_\lambda \subseteq S^d V_0$ is the variety of homogeneous polynomials of degree $d$ in two variables whose linear factors appear with multiplicities according to the partition $\lambda$:
\[
\calC_\lambda = \{ \ell_1^{\lambda_1} \cdots \ell_s^{\lambda_s} : \ell_j \in V_0 \} \subseteq S^d V_0.
\]
Coincident root loci have been studied classically, but their equations are known only in few cases \cite{HilbPowers, Ch04, AbCh07}. However, one can realize a system of equations of $\calTNS^\Delta_{\bfm,\bfn}$ as pull-back of equations of some specific coincident root loci $\calC_\lambda$. 

More precisely, there is a collection of partitions $\Lambda$ and polynomial maps (and in fact a collection of such) $\phi_\lambda : V_0 \otimes V_1 \otimes V_2 \to S^{d_\lambda} V_0$ for $\lambda \in \Lambda$ such that, if $T \in \calTNS^\Delta_{\bfm,\bfn}$ then $\phi_\lambda(T) \in \calC_\lambda$ for every $\lambda \in \Lambda$ and if $T$ is generic in $V_0 \otimes V_1 \otimes V_2$ then $\phi_\lambda(T) \notin \calC_\lambda$ for any $\lambda$. In particular, since each coefficient of the polynomial $\phi_\lambda(T)$ is a polynomial in the entries of $T$, the equations of $\calC_\lambda$, evaluated at $\phi_\lambda(T)$, give (some) equations for $\calTNS^\Delta_{\bfm,\bfn}$. This is proved in the following result.

\begin{theorem}\label{thm: equations from coincident root loci}
Let $\bfn = (2,n_1,n_2)$ and $\bfm = (m_{01} , m_{12}, m_{02})$ be an assignment of local and physical dimensions of the triangular network $\Delta$, with $n_i \leq m_{ij}m_{ik}$ with $\{i,j,k\} = \{0,1,2\}$. Write $k_i = m_{ij}m_{ik} - n_i$ for $i = 1,2$ and suppose $k_2 \leq k_1$. If 
\begin{equation}
    \tag{\raisebox{0.3mm}{\textreferencemark}}
m_{01} = m_{02} \quad \text{ and } \quad  k_2 \leq k_1 < m_{12} 
\end{equation}
set $m=m_{01} = m_{02}$. Let $\kappa = k_1 \vvirg m_{12}-2$. A system of equations of $\calTNS^\Delta_{\bfm,\bfn}$ is realized as pull-back of the equations of $\calC_{\lambda(\kappa)}$ where  
\[
\lambda(\kappa) = (\underbrace{m_{12} - \kappa  \vvirg m_{12} - \kappa }_{m \text{ times} }, \underbrace{1 \vvirg 1}_{(m-1)\kappa \text{ times}}).
\]
  
\end{theorem}
\begin{proof}
    Let $V_1', V_2'$ be vector spaces of dimension $n' =  m m_{12} - \kappa$. If $T \in \calTNS^{\Delta}_{\bfm, (2,n_1,n_2)}$, and $X_i \in \Hom(V_i,V_i')$, then by construction 
    \[
   (X_1 \otimes X_2) \cdot T \in \calTNS^{\Delta}_{\bfm, (2,n',n')}.
    \]
    In particular, equations for $\calTNS^{\Delta}_{\bfm, (2,n',n')}$ pull back to equations for $\calTNS^{\Delta}_{\bfm, (2,n_1,n_2)}$. We determine equations of $\calTNS^{\Delta}_{\bfm, (2,n',n')}$ as pull back of equations of $\calC_{\lambda(\kappa)}$. Note $n' = m(m_{12} - \kappa) + (m-1)\kappa$ equals the length of the partition $\lambda(\kappa)$.

Consider the map
    \[
   \phi_{\lambda(\kappa)} : V_0 \otimes V_1' \otimes V_2' \to S^{n'} V_0
    \]
    defined by sending $T \in V_0 \otimes V_1' \otimes V_2'$ to the determinant of the associated $n' \times n'$ matrix of linear forms in a fixed basis of $V_0$. By \autoref{thm: triangular subcritical}, if $T \in \calTNS^{\Delta}_{\bfm, (2,n',n')}$ is generic then $\phi_{\lambda(\kappa)}(T)$ has the $m$ factors $z_i = v_0 + \zeta_i v_1$ occurring with multiplicity $m_{12}-\kappa$ and $(m-1)\kappa$ factors occurring with multiplicity $1$. Therefore $\phi_{\lambda(\kappa)}(T) \in \calC_{\lambda(\kappa)}$ for every $T \in \calTNS^{\Delta}_{\bfm, (2,n',n')}$.

    If $T \in V_0\otimes V_1' \otimes V_2'$ is generic, then $\phi_{\lambda(\kappa)}(T)$ is a generic binary form of degree $n'$, hence $\phi_{\lambda(\kappa)}(T) \notin \calC_\lambda$.
\end{proof}
The equations of \autoref{thm: equations from coincident root loci} do not define $\calTNS^{\Delta}_{\bfm,\bfn}$ set-theoretically. They simply cut out the variety of pencils $T \in V_0 \otimes V_1 \otimes V_2$ whose generalized eigenvalues have the same multiplicities as the ones of tensors in $\calTNS^{\Delta}_{\bfm,\bfn}$. In particular, they do not detect the Jordan structure of the blocks $J_1(z_i)^{\boxplus m}$ and they vanish on any pencil where the block $J_1(z_i) ^{\boxplus m}$ is replaced by some other sum of Jordan blocks, for instance the single Jordan block $J_m(z_i)$. 

More refined equations can be obtained translating the Kronecker structure into polynomial equations. This is implicitly done in \cite{Gant59} and in some very special cases in \cite{Dimca82,Wall78} but it is not immediate to write the corresponding equations explicitly. In some cases, the Kronecker structure of the blocks can be detected by geometric invariants, such as the \emph{collineation varieties} of \cite{GeKe}; these provide additional equations for $\calTNS^{\Delta}_{\bfm,\bfn}$, but they are not enough to cut out $\calTNS^{\Delta}_{\bfm,\bfn}$ set-theoretically either. In the case where $\bfm = (2,m,2)$, set-theoretic equations for the tensor network variety can be obtained as pull-back of certain generalizations of secant varieties; they are discussed in \autoref{sec: TNS as intersection with sigma}.

We point out that in the special case $n_2 = n_1+1$, \autoref{thm: triangular subcritical} implies that the pencils $T \in \calTNS^{\Delta}_{\bfm,\bfn}$ have a non-trivial block structure, whereas from the discussion of \autoref{sec: pencils} the generic pencil in $V_0 \otimes V_1 \otimes V_2$ consists of a single indecomposable singular block $L_{n_1}$. In this case, \cite[Cor. 10.7.3]{DerWey} guarantees that the ring of invariant polynomials $\bbC[V_0 \otimes V_1 \otimes V_2]^{\SL(V_0) \times \SL(V_1) \times \SL(V_2)}$  has a single generator, whose zero set is the complement of the orbit of the generic element. In particular, such invariant vanishes on $\calTNS^{\Delta}_{\bfm,\bfn}$. This fact was recently observed in \cite{vdBCLNZ25}, where it was translated into the non-membership of some specific points in the moment polytope of the graph tensor $T(\Delta,\bfm)$. \autoref{thm: triangular subcritical} recovers this result providing a geometric interpretation of the invariant.

In fact, we expect the construction of \emph{Schofield invariants} for pencils of matrices, as described in \cite[Ch. 10]{DerWey}, and more explicitly for pencils in \cite{Schofield, SchvdB}, would provide additional interesting modules of equations. We illustrate this in the following remark.

\begin{remark}\label{remark: schofield 234}
Consider the case $\bfm = (2,2,2)$ and $(n_1,n_2) = (3,4)$. Then $\calTNS^\Delta_{\bfm,\bfn} \subseteq V_0 \otimes V_1 \otimes V_2$ is a subvariety of codimension $2$. In this case the equations of \autoref{thm: equations from coincident root loci} are trivial, because they rely on the restrictions of $\calTNS^\Delta_{\bfm,(2,3,4)}$ to $\calTNS^\Delta_{\bfm, (2,3,3)}$ which coincides with the ambient space. 

Fix an integer $q$ and let $U_1, U_2$ be vector spaces with $\dim U_1 = 4q$ and $\dim U_2 = 3q$. Let $S \in V_0^* \otimes U_1 \otimes U_2$ be a generic element. For every $T \in V_0 \otimes V_1 \otimes V_2$, let $S \contract T \in U_1 \otimes U_2 \otimes V_1 \otimes V_2$, which defines a linear map $F(S,T): U_1^* \otimes V_1^* \to U_2 \otimes V_2$, represented by a square matrix of size $12q$; this is an instance of the \emph{bridge map} studied in \cite{GLS24}. If $T$ is generic, this map is full rank; in particular, the determinant of the map $F(S,T)$ is a polynomial in the coefficients of $T$, invariant under the action of $\SL(V_0) \times \SL(V_1) \times \SL(V_2)$, called \emph{Schofield invariant}; then \cite[Thm. 2.3]{SchvdB} guarantees that this construction yields a system of generators for the ring of invariants on $V_0 \otimes V_1 \otimes V_2$; in the case $(n_1,n_2) = (3,4)$ it turns out that this ring has a single generator, arising as determinant of the $12 \times 12$ matrix obtained by any sufficiently general $S \in V_0^* \otimes U_1 \otimes U_2$ in the case $q=1$. 

If $T \in \calTNS^\Delta_{\bfm,\bfn}$, then by \autoref{thm: triangle mm1}, $T$ is unstable, hence all Schofield invariants vanish. But in fact, a direct computation shows that, for generic $S \in V_0^* \otimes U_1 \otimes U_2$ with $\dim U_1 = 4$, $\dim U_2 = 3$, the map $F(S,T) : U_1^* \otimes V_1^* \to U_2 \otimes V_2$ has rank $10$. In particular, the $11 \times 11$ minors of $F(S,T)$, for every fixed $S$, define a system of equations for $\calTNS^\Delta_{\bfm,\bfn}$. Computational experiments suggest that these equations set-theoretically define $\calTNS^\Delta_{\bfm,\bfn}$, but the size of the problem is already beyond reach for a symbolic proof. 
\end{remark}

In the next section, we provide a geometric interpretation for membership in $\calTNS^{\Delta}_{\bfm,\bfn}$, which can, in principle, be translated into polynomial equations.

\section{Geometry of the flattening image}\label{sec:5}

In this section we discuss a geometric interpretation for a tensor $T$ to belong to $\calTNS^\Delta_{\bfm,\bfn}$ for small values of $\bfm$ and $\bfn$: we focus on the pencil cases discussed in \autoref{section: pencils} and on the case $\bfm = (2,2,2)$, $\bfn = (3,3,3)$. In these cases, the membership in $\calTNS^\Delta_{\bfm,\bfn}$ can be interpreted in terms of special intersection properties of the image of the flattening map $T: V_0^* \to V_1 \otimes V_2$ with suitable subvarieties of $V_1 \otimes V_2$.

\subsection{Pencils and a geometric interpretation of \autoref{main theorem}}\label{sec: TNS as intersection with sigma}
A characterization for the membership of $T$ in $\calTNS^{\Delta}_{\bfm,\bfn}$ in the case of bond dimensions $(2,2,2)$ and physical dimensions $(2, 4,4)$ was given in \cite[Theorem 5.2]{BDG23}. We provide an upgrade of this construction in \autoref{thm: triangles with 222 corner}, in the setting of \autoref{222corner:geo:int}, and a more general version in \autoref{thm: triangle and Z variety}. Let $\sigma_{r}^{n_1 \times n_2} \subseteq \bbP ( \bbC^{n_1} \otimes \bbC^{n_2})$ be the variety of matrices of rank at most $r$. 

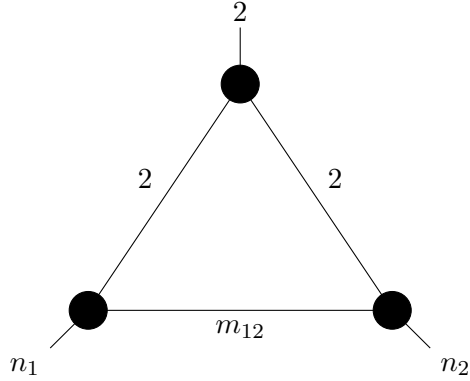
\begin{figure}[!h]
\centering
\begin{tikzpicture}[scale=.5]
\draw (0,0)--(4,6);
\draw (0,0)--(8,0);
\draw (8,0)--(4,6);
\node () at (1.5,3.5) {$2$};
\node () at (6.5,3.5) {$2$};
\node () at (4,-0.5) {$m_{12}$};
\draw (0,0)--(-1,-1);
\draw (4,6)--(4,7.5);
\draw (8,0)--(9,-1);
\draw[fill=black] (0,0) circle (.5cm);
\draw[fill=black] (4,6) circle (.5cm);
\draw[fill=black] (8,0) circle (.5cm);
\node[anchor = north east] at (-1,-1){$n_1$};
\node[anchor = north west] at (9,-1){$n_2$};
\node[anchor = south] at (4,7.4){$2$};
\node[anchor = north east] at (-1,-1){\phantom{$\bbC^{n_0}$}};
\node[anchor = north west] at (9,-1){\phantom{$\bbC^{n_1}$}};
\node[anchor = south] at (4,7.4){\phantom{$\bbC^{n_2}$}};
\node at (0,0){$0$};
\node at (4,6){$1$};
\node at (8,0){$2$};
\node at (-2,2){\phantom{$T = $}};
\end{tikzpicture}
\caption{Triangular graph with bond dimensions $\mathbf{m} = (2,m_{12},2)$ and $\mathbf{n} = (2,n_{1},n_2)$.}\label{222corner:geo:int}
\end{figure}

\begin{theorem}\label{thm: triangles with 222 corner}
Let $\bfn = (2,n_1,n_2)$ and $\bfm = (2 , m_{12}, 2)$ be an assignment of local and physical dimensions of the triangular network $\Delta$. Then
 \[
  \calTNS^\Delta_{\bfm,\bfn} = \bar{ \left\{ T \in V_0 \otimes V_1 \otimes V_2 : \begin{array}{l} \bbP ( \im (T : V_0^* \to V_1 \otimes V_2) )\text{ is a line intersecting} \\ \sigma_{m_{12}}^{n_1 \times n_2} \text{ in at least two points}\end{array} \right\}} .
 \]
\end{theorem}
\begin{proof}
Let 
\[
\calZ_{m_{12}}^{n_1 \times n_2} = \bar{ \left\{ T \in V_0 \otimes V_1 \otimes V_2 : \begin{array}{l} \bbP (\im (T : V_0^* \to V_1 \otimes V_2)) \text{ is a line intersecting} \\ \sigma_{m_{12}}^{n_1 \times n_2} \text{ in at least two points}\end{array} \right\}}.
\]
We are going to prove $\calTNS^\Delta_{\bfm,\bfn} = \calZ_{m_{12}}^{n_1 \times n_2}$.

By \autoref{thm: triangular subcritical}, generic elements of $\calTNS^\Delta_{\bfm,\bfn}$ can be normalized to be of the form
\[
T = ( J_1(v_0 + \zeta_1 v_1) \boxtimes \bfI_{m_{12} - k_1} ) \boxplus (J_1(v_0 + \zeta_2 v_1) \boxtimes \bfI_{m_{12} - k_1} ) \boxplus T_2
\]
with $T_2$ generic in $\bbC^2 \otimes \bbC^{k_1} \otimes \bbC^{2k_1 - k_2}$.

Let $\alpha \in V_0^*$ be defined by $\alpha_i(v_1) = -1$ and $\alpha_i(v_0) = \zeta_i$. Then $T(\alpha_i) = \bfI_{m_{12} - k_1} + T_2(\alpha_i)$ which is a matrix of rank at most $m_{12} - k_1 + k_1 = m_{12}$ because $k_2 \leq k_1$. This shows that the projective line $\bbP (\im (T : V_0^* \to V_1 \otimes V_2)) $ intersects $\sigma_{m_{12}}^{n_1 \times n_2}$ in at least two points, namely the two matrices $T(\alpha_0),T(\alpha_1)$. We deduce that $\calTNS^\Delta_{\bfm,\bfn} \subseteq \calZ_{m_{12}}^{n_1 \times n_2}$.

The other inclusion follows by a dimension argument. By \cite[Lemma 5.1]{BDG23} 
\[
\dim \calZ_{m_{12}}^{n_1 \times n_2} =  2 m_{12} (n_1 + n_2 - m_{12}) + 2 = 6m_{12}^2 - 2m_{12} k _1 - 2m_{12}k_2 + 2;
\]
we point out that the dimension recorded in \cite[Lemma 5.1]{BDG23} considers the tensor network variety in projective space, whereas in this paper we consider its affine cone in $V_0 \otimes V_1 \otimes V_2$. By \autoref{thm: triangular subcritical}, this dimension is the same as $\calTNS^\Delta_{\bfm,\bfn}$.
\end{proof}

\autoref{thm: triangles with 222 corner} yields a complete system of set-theoretic equations for $\calTNS^\Delta_{\bfm,\bfn}$ as pull-back of the equations of the Grassmann secant variety of $\sigma_{m_{12}}^{n_1 \times n_2}$ in $\Gr(2, V_1 \otimes V_2)$, in the sense of \cite{BalBerCatChi}. They are, however, not easy to write down explicitly. It would be interesting to determine such conditions in terms for the Schofield matrices described in \autoref{remark: schofield 234}. A similar result is given in \autoref{thm: triangle and Z variety} for the more general setting of \autoref{thm: triangular subcritical}.

\begin{theorem} \label{thm: triangle and Z variety}
 Let $\Delta$ be the triangle graph with bond dimensions $\bfm = (m,m_{12},m)$, and local dimensions $\bfn = (2,n_1, n_2)$ with $n_1 = m_{12}m-k_1$ and $n_2 = m_{12}m-k_2$. Then 
 \[
  \calTNS^\Delta_{\bfm,\bfn} = \bar{ \left\{ T \in \bbP ( V_0 \otimes V_1 \otimes V_2) : \begin{array}{l} \bbP ( \im (T : V_0^* \to V_1 \otimes V_2)) \text{ is a line intersecting} \\ \sigma_{m_{12}(m-1)}^{n_1 \times n_2} \text{ in at least $m$ distinct points}\end{array} \right\}} .
 \]
\end{theorem}

\begin{proof}
As in \autoref{thm: triangular subcritical}, without loss of generality assume $k_2 \leq k_1$ so that $n_1 \leq n_2$. Let  
\[
\mathcal{Z}=\overline{
\left\{
T\in\mathbb P(V_0\otimes V_1\otimes V_2):
\begin{array}{l}
\mathbb P(\mathrm{im}(T:V_0^*\to V_1\otimes V_2))\ \text{is a line intersecting}\\
\sigma_{m_{12}(m-1)}^{n_1\times n_2}\ \text{in at least $m$ distinct points}
\end{array}
\right\}
}.
\]
We are going to prove $\calTNS^\Delta_{\bfm,\bfn} = \mathcal{Z}$. 

The inclusion $\calTNS^\Delta_{\bfm,\bfn} \subseteq \mathcal{Z}$ follows from \autoref{thm: triangular subcritical}. Indeed, if $T$ is a generic element of $\calTNS^\Delta_{\bfm,\bfn}$, then by \autoref{thm: triangular subcritical}, we may choose coordinates so that 
\[
T = \bigboxplus_{i=1}^m (J_1(z_i))^{\boxplus (m_{12} -k_1)} \boxplus T_2
\]
where $T_2$ is generic in $\bbC^2 \otimes \bbC^{(m-1)k_1} \otimes \bbC^{(m-1)k_1 + (k_1 - k_2)}$ and $z_1 \vvirg z_m$ distinct elements of $V_0$. For every $i$, let $\alpha_i \in V_0^*$ be the unique element, up to scaling, such that $\alpha_i( v_0 +\zeta_i v_1) = 0$. Then 
\[
T(\alpha_j) = [\textstyle \bigoplus_{ \substack{ i = 1 \vvirg m \\ j \neq i}}  \lambda_i \bfI_{m_{12} - k_1}] \oplus T_2(\alpha_j),
\]
where $\lambda_i = \alpha_j(z_i)$ are nonzero scalars.

By genericity of $T_2$, we have $T_2(\alpha_j)$ is a matrix of rank $(m-1)k_1$. Therefore 
\[
\rank ( T(\alpha_j)) = (m-1)(m_{12} - k_1) + (m-1)k_1 = (m-1)m_{12}.
\]
This shows that there are $m$ distinct elements $\alpha_1,\ldots,\alpha_m \in \mathbb P(V_0^*)$ such that $T(\alpha_j) \in \sigma_{m_{12}(m-1)}^{n_1\times n_2}$, which guarantees $T\in\mathcal Z$.

To prove the inclusion $\mathcal{Z} \subseteq \calTNS^\Delta_{\bfm,\bfn}$, let $T\in\mathcal Z$ satisfy the property that $\bbP (\im (T : V_0^* \to V_1 \otimes V_2))$ is a line intersecting $\sigma_{m_{12}(m-1)}^{n_1\times n_2}$ in $m$ distinct points.

We use the technical \autoref{lemma: pencils with rank drops}, proved below, to show that $T \in \calTNS^{\Delta}_{\bfm,\bfn}$. For $i = 1,2$, let $V_i' \subseteq V_i$ be subspaces such that $T$ is concise in $V_0 \otimes V_1' \otimes V_2'$. The statement that $\bbP (\im (T : V_0^* \to V_1 \otimes V_2))$ intersects $\sigma_{m_{12}(m-1)}^{n_1\times n_2}$ in at least $m$ distinct points is equivalent to the statement that $\bbP (\im (T : V_0^* \to V_1' \otimes V_2'))$ intersects $\sigma_{m_{12}(m-1) - (n_1-n_1')}^{n_1'\times n_2'}$ in at least $m$ distinct points.

The value $m_{12}(m-1) - (n_1-n_1')  = n_1' - [n_1 - m_{12}(m-1)]$ plays the role of the value $n_1 - s$ in \autoref{lemma: pencils with rank drops}. Write $T'$ for the tensor $T$ regarded as a tensor in $V_0 \otimes V_1' \otimes V_2'$. By \autoref{lemma: pencils with rank drops}, there exist distinct $z_1 \vvirg z_m$ such that $T'$ has at least $n_1 - m_{12}(m-1) = m_{12} - k_1$ Jordan blocks with eigenvalue $z_j$. Because of the normal form of \autoref{thm: triangle mm1}, such pencils are degenerations of suitable concise pencils in $\calTNS^{\Delta}_{\bfm,\bfn}$. This shows the inclusion $\mathcal{Z} \subseteq \calTNS^\Delta_{\bfm,\bfn}$ and concludes the proof.
\end{proof}
The following technical lemma is used in the proof of \autoref{thm: triangle and Z variety}.
\begin{lemma} \label{lemma: pencils with rank drops}
Let $n_1 \leq n_2$ and let $T \in V_0 \otimes V_1 \otimes V_2$ be a concise tensor such that $\bbP (\im (T : V_0^* \to V_1 \otimes V_2))$ is a line intersecting $\sigma_{n_1-s}^{n_1 \times n_2}$ in at least $p$ distinct points. Then there exist distinct $z_1 \vvirg z_p \in V_0$ such that the Kronecker decomposition of $T$ contains at least $s$ Jordan blocks with eigenvalue $z_i$.
\end{lemma}
\begin{proof}

We prove the statement proceeding by induction on $p$. If $p = 0$ there is nothing to prove. Let $p \geq 1$ and suppose the statement holds for $p-1$.

Let $\alpha_1 \vvirg \alpha_p \in V_0^*$ be such that $\rank (T(\alpha_i)) \leq n_1-s$ and let $z_1 \vvirg z_p \in V_0$ be the unique, up to scaling, elements such that $\alpha_j(z_j) = 0$. The condition that $\rank ( T(\alpha_p) ) \leq n_1-s$ guarantees that $T$ has at least $s$ Jordan blocks with eigenvalue $z_p$. Indeed singular blocks $L_k$ and $R_k$ with $k \geq 1$ are full rank at every point; moreover, since $T$ is concise, no singular blocks $L_0$ or $R_0$ appear in the Kronecker decomposition of $T$. Therefore the only blocks that can contribute to the rank drop of $T(\alpha_p)$ are Jordan blocks with eigenvalue $z_p$. Since every Jordan block contributes with a $1$-dimensional kernel, hence a drop by $1$, we conclude that there must be $s$ such blocks. 

This shows $T = J_{\bullet}(z_p) \boxplus T'$ where $J_\bullet$ is the sum of the (at least $s$) Jordan blocks with eigenvalue $z_p$ and $T'$ is a matrix pencil with the property that its rank drops by at least $s$ in at least $p-1$ points. By the induction hypothesis $T'$ satisfies the desired statement and this concludes the proof.
\end{proof}

\subsection{Triangular tensor network with larger physical dimensions}

In this section, we consider the case of the triangular tensor network with bond dimensions $\bfm = (2,2,2)$ and physical dimensions $\bfn = (3,3,3)$. This case does not fall into the framework of \autoref{main theorem}; however, the characterization of the tensor network variety can be achieved with methods similar to the ones of the proof of \autoref{main theorem}.

Fix $V_0,V_1,V_2$ be vector spaces with $\dim V_i = 3$. A tensor $T \in V_0 \otimes V_1 \otimes V_2$ defines three determinantal varieties
\[
\scrE_i(T) = \{ \alpha \in \bbP V_i^* : \det(T(\alpha_i)) = 0 \} \subseteq \bbP V_i.
\]
The variety $\scrE_i(T)$ is a cubic curve if $T(V_i^*)$ contains matrices of rank $3$ and it is the whole $\bbP V_i^*$ otherwise. It was shown in \cite{Ng95} that if $\scrE_i(T)$ is a smooth cubic curve for one $i$ then it is for all three and the three curves $\scrE_0(T),\scrE_1(T),\scrE_2(T)$ are isomorphic as elliptic curves. In fact, it was shown that in this case the tensor $T$ is uniquely determined, up to the action of $\GL(V_0) \times \GL(V_1) \times \GL(V_2)$, by the three embeddings of the elliptic curve in $\bbP^2$, modulo a natural equivalence relation; see also \cite{Bea00} for a much more general form of this result. A more refined classification of tensors in $V_0 \otimes V_1 \otimes V_2$, also including the tensors for which the curves $\scrE_i(T)$ are singular or reducible, was given in \cite{Nur00b, Nur00} and the more recent \cite{DitDeGrMar} reproposes this classification fixing some inconsistencies.

The tensor network $\calTNS^{\Delta}_{\bfm,\bfn}$ in the case $\bfm = (2,2,2)$ and $\bfn = (3,3,3)$ can be characterized using these classifications. We record four tensors appearing in \cite{DitDeGrMar} which play a role in the proof of this result:
{\small
\[
\begin{aligned}
T^{(\lambda)}_{ \mathrm {II.2}} = &a_0 \otimes b_0 \otimes c_0 + a_1 \otimes b_1 \otimes c_1 + a_2 \otimes b_2 \otimes c_2 + \lambda ( a_0 \otimes b_1 \otimes c_2 + a_1 \otimes b_2 \otimes c_0 + a_2 \otimes b_0 \otimes c_1) \\
 &+ a_0\otimes b_2 \otimes c_1;\\
T_{ \mathrm {III.2}} = &a_0 \otimes b_0 \otimes c_0+a_1 \otimes b_1 \otimes c_1+a_2 \otimes b_2 \otimes c_2  \\
&+a_0 \otimes b_1 \otimes c_2+a_0 \otimes b_2 \otimes c_1+a_1 \otimes b_0 \otimes c_2;\\
T_{ \mathrm {IV.2}} = &a_0 \otimes b_1 \otimes c_2 + a_1 \otimes b_2 \otimes c_0 + a_2 \otimes b_0 \otimes c_1 - a_1 \otimes b_0 \otimes c_2 - a_2 \otimes b_1 \otimes c_0 - a_0 \otimes b_2 \otimes c_1 \\
       &+ a_0 \otimes b_0 \otimes c_1 + a_0 \otimes b_1 \otimes c_0 + a_1 \otimes b_0 \otimes c_0 + a_0 \otimes b_2 \otimes c_2 + a_2 \otimes b_0 \otimes c_2 + a_2 \otimes b_2 \otimes c_0 ;\\
T_{ \mathrm {0.2}} = & a_0 \otimes b_1 \otimes c_2+a_0 \otimes b_2 \otimes c_1+a_1 \otimes b_0 \otimes c_2+a_1 \otimes b_1 \otimes c_0+a_1 \otimes b_1 \otimes c_1+a_2 \otimes b_0 \otimes c_0.
\end{aligned}
\]}
The indices correspond to the families in the classification of \cite{DitDeGrMar}; the first index in roman numerals correspond to the family in \cite{Nur00}, and the second index is the position of the tensor in such family: for instance, $T_{ \mathrm {III.2}}$ is the second tensor of Nurmiev's third family, which appears in the second row of \cite[Table IV]{DitDeGrMar}. The index $0$ corresponds to the family of orbits in the nullcone for the action of $\SL(V_0) \times \SL(V_1) \times \SL(V_2)$, which appears in  \cite[Table I]{DitDeGrMar}.
\begin{theorem}\label{thm: triangle 222 333}
 Let $\Delta$ be the triangle graph with bond dimensions $\bfm = (2,2,2)$ and local dimensions $\bfn = (3,3,3)$. Let $T \in V_0 \otimes V_1 \otimes V_2$. The following are equivalent: 
\begin{enumerate}
\item $T \in \calTNS^{\Delta}_{\bfm,\bfn}$;
\item the defining polynomials of $\scrE_0(T), \scrE_1(T), \scrE_2(T)$ are reducible;
\item $T$ is in the orbit-closure of the family of tensors $T^{(\lambda)}_{ \mathrm {II.2}}$.
\end{enumerate} 
\end{theorem}
\begin{proof}
We first prove the implication (1) $\Rightarrow$ (2). Let $T \in \calTNS^{\Delta}_{\bfm,\bfn}$ be a generic element, that is a restriction of the graph tensor 
\[
T(\Delta, (2,2,2) ) = \sum_{i,j,k=0}^1 v^{(0)}_{ij} \otimes v^{(1)}_{jk} \otimes v^{(2)}_{ki} \in W_0 \otimes W_1 \otimes W_2.
\]
 The image of the flattening $W_0^* \to W_1 \otimes W_2$ of the graph tensor is the linear space of matrices defined by 
\[
M_0 = I_2 \boxtimes \left( \begin{array}{cc}
    v_{00} & v_{0 1} \\
    v_{10} & v_{11}
\end{array}\right) \subseteq W_1 \otimes W_2,
\]
where we indicate $v_{ij} = v^{(0)}_{ij}$ suppressing the superscript. The determinant of this matrix is $\det(M_0) = (v_{00} v_{11} - v_{01} v_{10})^2$ which is an element of $S^4 W_0$. Now let $T = X_0 \otimes X_1 \otimes X_2 (T(\Delta, (2,2,2)))$ be a generic element of $\calTNS^{\Delta}_{\bfm,\bfn}$.

Then the defining equation of $\scrE_0(T)$ is a cubic arising as a linear combination of size $3$ minors of $M_0$ restricted via the map $X_0$. By construction, the size $3$ minors of $M_0$ are multiples of $v_{00} v_{11} - v_{01} v_{10}$; in particular they are reducible. Since restriction of reducible polynomials are reducible, the same holds for the defining equation of $\scrE_0(T)$. The same argument applies to $\scrE_1(T)$ and $\scrE_2(T)$, proving the inclusion $\calTNS^{\Delta}_{\bfm,\bfn} \subseteq \calZ$.

The implication (2) $\Rightarrow$ (3) follows from the classification of \cite{Nur00b, Nur00, DitDeGrMar}. Direct computation on the classification guarantees that the only elements for which the curves $\scrE_i(T)$ are reducible are, up to changing coordinates
\begin{itemize}
\item $T^{(\lambda)}_{ \mathrm {II.2}}$ and its degenerations;
\item $T_{ \mathrm {III.2}}$, its permutations and its degenerations;
\item $T_{ \mathrm {IV.2}}$ and its degenerations;
\item $T_{ \mathrm {0.2}}$ and its degenerations.
\end{itemize}
Consider the variety 
\[
\calT = \overline{ \GL(V_0) \times \GL(V_1) \times \GL(V_2) \cdot \{ T^{(\lambda)}_{ \mathrm{II.2}} : \lambda \in \bbC \}}.
\]
We are going to show that $T_{ \mathrm{III.2}}, T_{ \mathrm{IV.2}},T_{ \mathrm{0.2}}$ belong to $\calT$. This is verified explicitly by determining suitable curves $g_{i}(\eps) \in \GL(V_i)$ and $\lambda_\eps \in \bbC$ such that the limit $\lim _{\eps \to 0} [(g_0(\eps) \otimes g_1(\eps) \otimes g_2(\eps)) T^{(\lambda_\eps)}_{ \mathrm{II.2}}]$ equals the desired tensor.

For the case of $T_{ \mathrm{III.2}}$, consider $\lambda_\eps = \eps^{-1}$ and 
\[
\begin{array}{rrlcrrlcrrl}
    g_0(\eps) : & a_0 & \mapsto \eps^{-1} a_0  & \quad & g_1(\eps) : & b_{0} & \mapsto  b_2 & \quad & g_2(\eps) :  & c_{0} & \mapsto \eps c_1 \\
      & a_{1} & \mapsto \eps^{-1} a_1 & \quad & & b_{1} & \mapsto \eps b_0 & \quad & & c_{1} & \mapsto c_2 \\ 
      & a_{2} & \mapsto \eps a_2 & \quad & & b_{2} & \mapsto \eps b_1 & \quad & & c_{2} & \mapsto \eps c_0 
\end{array}
\]
Then 
\[
 T_{ \mathrm{III.2}} = \lim_{\eps \to 0} [(g_0(\eps) \otimes g_1(\eps) \otimes g_2(\eps)) \cdot  T^{(\lambda_\eps)}_{ \mathrm{II.2}} ],
\]
showing $T_{ \mathrm{III.2}} \in \calT$. 

For the case of $T_{ \mathrm{IV.2}}$, consider $\lambda_\eps = -1 + \eps^2$ and
\begin{gather*}
\begin{array}{rrlcrrl}
    g_0(\eps) : & a_0 & \mapsto 2a_0 + \eps^4 a_1 - \eps^2 a_2  & \quad & 
    g_1(\eps) : & b_{0} & \mapsto  -\frac{2}{3}\eps^{-1} b_0  + \eps b_2 \\ 
    & a_{1} & \mapsto -\frac{2}{3}\eps^{-1} a_0  + \eps a_2 & \quad & 
    & b_{1} & \mapsto 2\eps^{-1} b_0 + \eps b_2 \\
    & a_{2} & \mapsto 2\eps^{-1} a_0 + \eps a_2 & \quad & 
    & b_{2} & \mapsto 2b_0 + \eps^4 b_1 - \eps^2 b_2  \\
\end{array}\\
\begin{array}{rrl}
    g_2(\eps) : & c_0 & \mapsto 2\eps^{-1} c_0 + \eps c_2  \\ 
    & c_{1} & \mapsto  2c_0 + \eps^4 c_1 - \eps^2 c_2  \\
    & c_{2} & \mapsto -\frac{2}{3}\eps^{-1} c_0  + \eps c_2 
\end{array}
\end{gather*}
Then 
\[
T_{ \mathrm{IV.2}} = \lim_{\eps \to 0} [(g_0(\eps) \otimes g_1(\eps) \otimes g_2(\eps)) \cdot  T^{(\lambda_\eps)}_{ \mathrm{II.2}} ].
\]
Finally, we show that $T_{ \mathrm{0.2}}$, is a degeneration of (a permutation of) $T_{ \mathrm{III.2}}$. Consider the tensor obtained from $T_{ \mathrm{III.2}}$ by permuting the second and third tensor factors and then applying the change of coordinates given by 
\[
\begin{array}{rrlcrrlcrrl}
    h_0 : & a_0 & \mapsto a_1 - a_0  & \quad & h_1 : & b_{0} & \mapsto  b_0 & \quad & h_2 :  & c_{0} & \mapsto c_2-c_1 \\
      & a_{1} & \mapsto a_0  & \quad & & b_{1} & \mapsto b_0 + 2b_1+ b_2 & \quad & & c_{1} & \mapsto c_0 + c_2 \\ 
      & a_{2} & \mapsto a_0-2a_1+a_2 & \quad & & b_{2} & \mapsto b_0 + b_1 & \quad & & c_{2} & \mapsto  c_0  .
\end{array}
\]
The resulting tensor is 
{ \small
\[
T' = a_0 \otimes b_1\otimes c_2+a_0\otimes b_2\otimes c_1+a_1\otimes b_0\otimes c_2+a_1\otimes b_1\otimes c_0+a_1\otimes b_1\otimes c_1+a_1\otimes b_2\otimes c_0+a_2\otimes b_0\otimes c_0+a_2\otimes b_1\otimes c_0.
\]
}
Now consider the degeneration defined by 
\[
\begin{array}{rrlcrrlcrrl}
    g_0(\eps) : & a_0 & \mapsto a_0  & \quad & g_1(\eps) : & b_{0} & \mapsto \eps^{-2} b_0 & \quad & g_2(\eps) :  & c_{0} & \mapsto c_0 \\
      & a_{1} & \mapsto \eps a_1 & \quad & & b_{1} & \mapsto \eps^{-1} b_1 & \quad & & c_{1} & \mapsto c_1 \\ 
      & a_{2} & \mapsto \eps^2 a_2 & \quad & & b_{2} & \mapsto  b_2 & \quad & & c_{2} & \mapsto \eps c_2 
\end{array}
\]
and conclude observing that 
\[
T_{ \mathrm{0.2}} = \lim_{\eps \to 0} [(g_0(\eps) \otimes g_1(\eps) \otimes g_2(\eps)) \cdot  T' ] .
\]

The implication (3) $\Rightarrow$ (1) can be verified exhibiting an explicit restriction of $T(\Delta, (2,2,2))$ to $T^{(\lambda)}_{ \mathrm{II.2}}$. Write $T(\Delta,(2,2,2)) = \sum_{i,j,k=0}^1 v^{(0)}_{ij} \otimes v^{(1)}_{jk} \otimes v^{(2)}_{ki}$ for the graph tensor. For a given $\lambda$, let $\mu = (1+\lambda^3)^{1/3}$, and consider the element $(X_0,X_1,X_2) \in \Hom(W_0,V_0) \times \Hom(W_1,V_1) \times \Hom(W_2,V_2)$ defined by
\[
\begin{array}{rrlcrrlcrrl}
    X_0 : & v^{(0)}_{00} & \mapsto \mu^{-1} a_0  & \quad & X_1 : & v^{(1)}_{00} & \mapsto \mu^{-1}b_2 & \quad & X_2 :  & v^{(2)}_{00} & \mapsto \mu^{-1}c_1 \\
      & v^{(0)}_{01} & \mapsto \mu a_1 & \quad & & v^{(1)}_{01} & \mapsto \mu b_0 & \quad & & v^{(2)}_{01} & \mapsto \mu c_2 \\ 
      & v^{(0)}_{10} & \mapsto a_2 & \quad & & v^{(1)}_{10} & \mapsto b_1 & \quad & & v^{(2)}_{10} & \mapsto c_0 \\ 
      & v^{(0)}_{11} & \mapsto \mu^{-1}\lambda a_0 & \quad & & v^{(1)}_{11} & \mapsto \mu^{-1} \lambda b_2 & \quad & & v^{(2)}_{11} & \mapsto \mu^{-1} \lambda  c_1 \\
\end{array}
\]
Then one can verify $T^{(\lambda)}_{\mathrm{II.2}} = (X_0,X_1,X_2) T(\Delta,(2,2,2))$. This proves the last implication and concludes the proof. 
\end{proof}

\autoref{thm: triangle 222 333} yields a complete system of set-theoretic equations for $\calTNS^{\Delta}_{(2,2,2),(3,3,3)}$ as the pull-back of the equations defining the variety of reducible cubics in $S^3 V_0 , S^3 V_1 , S^3 V_2$. These are equations of degree $8$ in the coefficients of the cubic form $f \in S^3 \bbC^3$, arising as maximal minors of the map 
\begin{equation}\label{eq: ruppert map}
\begin{aligned}
\rho: \fraksl_3 \to S^3 \bbC^3 \\ 
X \mapsto X.f,
\end{aligned}
\end{equation}
where $X.f$ is the image of $f$ via the Lie algebra action of $X \in \fraksl_3$ on $S^3 \bbC^3$; they are the first instance of \emph{Ruppert's equations}, from \cite{Rup85}. The minors of the Ruppert map in \eqref{eq: ruppert map} yields a system of $35$ equations of degree $8$, see, e.g., \cite[Remark 3.15]{FGOV25}. Since every cubic $\scrE_i(T)$ has coefficients of degree $3$ in the coefficients of $T$, one obtains a system of (at most) $3 \times 35 = 105$ equations of degree $24$ which set-theoretically cut out $\calTNS^{\Delta}_{(2,2,2),(3,3,3)}$. We record this fact in the following corollary.
\begin{corollary}\label{cor: ruppert for TNS}
     Set-theoretic equations for $\calTNS^{\Delta}_{(2,2,2),(3,3,3)}$ are given by a system of degree $24$ polynomials on $V_0 \otimes V_1 \otimes V_2$ arising as maximal minors of three $10 \times 8$ matrices whose entries are cubic forms in the coefficients of $T \in V_0 \otimes V_1 \otimes V_2$.
\end{corollary}

\section{Extension of arbitrary graphs via pencils}\label{sec:6}
In this section, we provide an analog of \autoref{thm: triangle mm1} for arbitrary networks. The tensor network discussed in \autoref{main theorem} can be thought of as the network on the graph $\Gamma' = (\{1,2\},\{(1,2)\})$, consisting of a single segment, augmented with a node $0$ and two edges $(0,1)$ and $(0,2)$, with physical dimension $n_0 = 2$ and bond dimensions $m_{01}$ and $m_{02}$ on the two additional edges. 

One may perform this operation starting from any arbitrary network $\Gamma' = (v(\Gamma'), e(\Gamma'))$ with two distinguished nodes $1,2$ and study the relations between tensor network varieties on $\Gamma'$ and on its augmentation. In general, controlling the resulting tensor network variety is more challenging, because there is no analog of the Kronecker normal forms of matrix pencils.  However, in the case where all vertices are critical except for the new vertex $0$, we can provide a characterization similar to the one of \autoref{thm: triangle mm1}. 

Given a network $\Gamma'$ with $d$ nodes $[d] = \{1 , \vvirg , d\}$, the \emph{augmentation} of $\Gamma$ at nodes $\{1,2\}$ is the graph $\Gamma$ with vertex set $\bfv(\Gamma) = \{0\} \cup [d]$ and edge set $\bfe(\Gamma) = \bfe(\Gamma') \cup \{ (0,1),(0,2)\}$; see \autoref{fig:augmentation-pin222}.

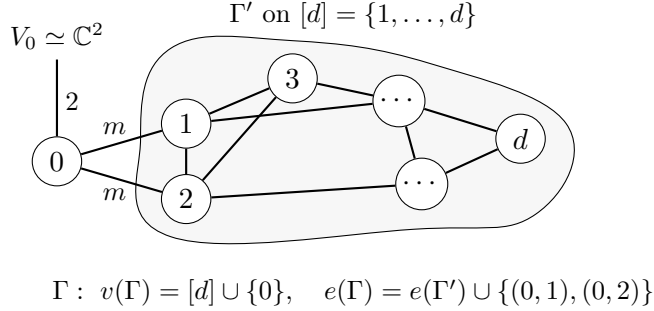
\begin{figure}[h!]
\centering
\begin{tikzpicture}[
  vertex/.style={circle,draw,fill=white,inner sep=1.6pt,minimum size=6.5mm},
  pinv/.style={circle,draw,fill=white,inner sep=1.6pt,minimum size=6.5mm},
  edge/.style={thick},
  lab/.style={font=\small}
]

% place some nodes
\node[vertex] (v1) at (-2.2, 0.4) {$1$};
\node[vertex] (v2) at (-2.2,-0.6) {$2$};

\node[vertex] (v3) at (-0.8, 1.0) {$3$};
\node[vertex] (v4) at ( 0.6, 0.7) {$\cdots$};
\node[vertex] (v5) at ( 0.9,-0.4) {$\cdots$};
\node[vertex] (vd) at ( 2.2, 0.2) {$d$};

% draw some edges to suggest an arbitrary connected graph
\draw[edge] (v1) -- (v2);
\draw[edge] (v1) -- (v3);
\draw[edge] (v1) -- (v4);
\draw[edge] (v2) -- (v3);
\draw[edge] (v2) -- (v5);
\draw[edge] (v3) -- (v4);
\draw[edge] (v4) -- (v5);
\draw[edge] (v4) -- (vd);
\draw[edge] (v5) -- (vd);

% a light background "cloud" around Gamma'
\begin{pgfonlayer}{background}
  \draw[fill=black!3] plot[smooth cycle, tension=0.7] coordinates {
     (-2.8, 0.2) (-2.4, 1.4) (-0.8, 1.5) (1.0, 1.2) (2.8, 0.4) (2.4, -0.8) (0.8, -1.0) (-2.4, -1.1)
  };
  \node[inner sep=10pt,
        fit=(v1)(v2)(v3)(v4)(v5)(vd)] (box) {};
\end{pgfonlayer}
\node[lab] at ($(box.north)+(0,0.15)$) {$\Gamma'$ on $[d]=\{1,\dots,d\}$};

% --- Attach the (2,2,2)-pin at the edge (1,2)
\node[pinv] (v0) at (-3.9,-0.1) {$0$};

\draw[edge] (v0) -- (v1)
  node[pos=0.55, above left=2pt, lab] {}
  node[pos=0.78, above left=1pt, xshift=-2pt, yshift=-4pt, lab] {$m$};

\draw[edge] (v0) -- (v2)
  node[pos=0.55, below left=2pt, lab] {}
  node[pos=0.78, below left=1pt, xshift=-2pt, yshift=4pt, lab] {$m$};

% physical leg at 0
\draw[edge] (v0) -- ++(0,1.35)
  node[above,lab] {$V_0\simeq \mathbb{C}^2$};
\node[lab] at ($(v0)+(0.2,0.8)$) {$2$};

% overall label Gamma
\node[lab] at ($(box.south)+(0,-0.55)$)
{$\Gamma:\ v(\Gamma)=[d]\cup\{0\},\quad e(\Gamma)=e(\Gamma')\cup\{(0,1),(0,2)\}$};

\end{tikzpicture}
\caption{Augmentation of a generic network $\Gamma'$ at nodes $\{1,2\}$, with bond dimensions $m$ on the new edges and physical dimension $2$ at the new vertex.}
\label{fig:augmentation-pin222}
\end{figure}

\begin{theorem}\label{augmentation:theorem}
Let $\Gamma'=(\bfv(\Gamma'),\bfe(\Gamma'))$ be a network with vertex set $\bfv(\Gamma')=[d]=\{1,\dots,d\}$ and let $\bfm'$ be an assignment of bond dimensions on $\Gamma'$. Let $\Gamma$ be the augmentation of $\Gamma'$ at nodes $\{1,2\}$. Let $\bfm$ be an assignment of bond dimensions on $\Gamma$ with 
\[
m_{e} = m'_e \text{ for all } e \in \bfe(\Gamma'), \qquad m_{01} = m_{02} = m.
\]
Let $\bfn$ be an assignment of physical dimension on $\Gamma$ where $n_0 = 2$, and for every $j \neq 0$ $n_j=\prod_{\substack{e\ni j}} m_e$ is the critical dimension. Then $(\Gamma,\bfm,\bfn)$ is defective and
\[
\delta^{\Gamma}_{\bfm,\bfn} = \operatorname{fiberdefect}^{\Gamma}_{\bfm,\bfn} = (m-1)(m_{12}^2-1).
\]
Moreover, a generic tensor $T\in\calTNS^\Gamma_{\bfm,\bfn}$, has the form
\[
T = (\textstyle \bigoplus_{i=1}^m J_1(z_i) )  \boxtimes T(\Gamma',\bfm'),
\]
for generic $z_1 \vvirg z_m \in V_0\simeq \bbC^2$, up to the action of $\GL(V_1) \ttimes \GL(V_d)$. Here $(\bigoplus_{i=1}^m J_1(z_i) )$ is regarded as a tensor in $V_0 \otimes V_1'' \otimes V_2'' \otimes \bbC^1 \ootimes \bbC^1$ and $T(\Gamma',\bfm') \in \bbC^1 \otimes V_1' \otimes V_2' \otimes V_3 \ootimes V_d$.
\end{theorem}

\begin{proof}
The proof is similar to the one of \autoref{thm: triangle mm1}.

We first determine the normal form of the generic element of $\calTNS^\Gamma_{\bfm , \bfn}$. For each vertex $i$ set
\[
W_i  := \bigotimes_{(i,j) \in\bfe(\Gamma)} E_{ij}.
\]
where $E_{ij} = E_{ji}^*$ is the $m_{ij}$-dimensional bond space associated to the edge $(i,j)$. So the graph tensor $T(\Gamma,\bfm)$ is an element of $W_0 \ootimes W_d$. 

Let $(X_0 \vvirg X_d)$ be a generic element in $\prod_{i=0}^d \Hom(W_i,V_i)$ so that $T = (X_0 \ootimes X_d) T(\Gamma,\bfm)$ is a generic element of $\calTNS^\Gamma_{\bfm , \bfn}$. Because of the assumption that all vertices except $0$ are critical, we have $\dim V_j = \dim W_j$; by genericity $X_1 \vvirg X_d$ are invertible, so they define isomorphisms $W_j \simeq V_j$. Via these isomorphisms $T$ may be regarded as an element of $V_0 \otimes W_1 \ootimes W_d$, and it is uniquely determined, up to the action of $\prod_{j=1}^d \GL(V_j)$, by $X_0 \in \Hom(W_0,V_0)$.

Now, $X_0 \in \Hom(W_0, V_0) = V_0 \otimes E_{01}^* \otimes E_{02}^*$ is a matrix pencil, with $\dim E_{01} = \dim E_{02} = m$. By genericity, it has the form $X_0 = \bigoplus_{i=1}^m J_1(z_i)$ for generic $z_i \in V_0$. Contracting $X_0$ with the graph tensor $T(\Gamma,\bfm)$, we obtain 
\[
T = X_0 (T (\Gamma,\bfm)) = X_0(T(\Delta, (m,1,m)))  \boxtimes T(\Gamma',\bfm') = \left(\textstyle \bigoplus_{i=1}^m J_1(z_i)\right) \boxtimes T(\Gamma',\bfm'),
\]
which is the desired normal form. In the following, write $z_i = v_0 + \zeta_i v_1$ for generic scalars $\zeta_i \in \bbC$ and a fixed basis $\{v_0,v_1\}$ of $V_0$. Let $A(\bfzeta) := ( \bigoplus_{i=1}^m J_1(v_0 + \zeta_i v_1)) \boxtimes T(\Gamma',\bfm')$ for $\bfzeta = (\zeta_1 \vvirg \zeta_m) \in \bbC^m$.

As in the proof of \autoref{thm: triangle mm1}, this construction allows one to define a reduced version of the map $\Phi$ that parametrizes the tensor network variety; this is given by 
\begin{align*}
\tilde{\Phi} : \bbC^m \times \GL(W_1) \ttimes \GL(W_d) &\to V_0 \otimes W_1 \ootimes W_d \\
 (\bfzeta, X_1 \vvirg X_d)  & \mapsto ( X_1 \ootimes X_d) (A(\bfzeta)).
\end{align*}
The image of $\tilde{\Phi}$ is dense in $\calTNS^\Gamma_{\bfm , \bfn}$, identified as a subvariety of $V_0 \otimes W_1 \ootimes W_d$ via fixed isomorphisms $W_j \simeq V_j$ for $j \neq 0$. Therefore the dimension of the fiber of $\tilde{\Phi}$ determines the dimension of the tensor network variety. 

We first argue that for a generic $T \in \calTNS^\Gamma_{\bfm , \bfn}$, every element $(\bfzeta, X_1 \vvirg X_d) \in \tilde{\Phi}^{-1}(T)$ has the same $\bfzeta$. To see this, regard $W_1 \ootimes W_d$ as a space of matrices, $W_1 \otimes (W_2 \ootimes W_d)$: the graph tensor $T(\Gamma',\bfm')$, under this regrouping, is simply a matrix of rank $\dim W_1$. The tensor $T \in V_0 \otimes W_1 \ootimes W_d$ may be regarded as a matrix pencil in $V_0 \otimes W_1 \otimes (W_2 \ootimes W_d)$ which, by the genericity of $X_0$, has the form described in \autoref{thm: triangle mm1}, $( \bigoplus_{i=1}^m J_1(v_0 + \zeta_i v_1))^{\oplus (\dim W_1')}$. Now, the map $\tilde{\Phi}$ is a restriction of the analogous map considered in \autoref{thm: triangle mm1}, where the second and third tensor factors are replaced by $W_1$ and $W_2 \ootimes W_d$. Therefore, by the same argument as in \autoref{thm: triangle mm1}, which relies on the fact that $\GL(W_1), \GL(W_2 \ootimes W_d)$ do not change the generalized eigenvalues of the pencil, the $\bfzeta$ associated to any element in the fiber $\tilde{\Phi}^{-1}(T)$ must be the same.

We deduce that the fiber $\tilde{\Phi}^{-1}(T)$ is isomorphic, as a subvariety, to the stabilizer of $T$ under the action of $\GL(W_1) \ttimes \GL(W_d)$. In particular, its dimension coincides with the dimension of the annihilator
\[
\ann_{\frakgl(W_1) \ttimes \frakgl(W_d)} (T).
\]
 Without loss of generality, we may assume $T = A(\bfzeta)$ with generic $\bfzeta$. By construction $A(\bfzeta)$ has the structure of a sum as in \autoref{lemma: ann for block sum pencils}. Similarly to \autoref{thm: triangle mm1}, the spaces $M_1, M_2$ of \autoref{lemma: ann for block sum pencils} are trivial, because regarding $A(\bfzeta)$ as a pencil and in the previous part of the argument, the generalized eigenvalues of the different Jordan blocks are distinct. This implies that
 \begin{multline*}
 \ann_{\frakgl(W_1) \ttimes \frakgl(W_d)} (T) \subseteq \\
 \ann_{\frakg_1}(J_1(v_0 + \zeta_1 v_1) \boxtimes T(\Gamma',\bfm')) + \cdots + \ann_{\frakg_m}(J_1(v_0 + \zeta_m v_1) \boxtimes T(\Gamma',\bfm')).
 \end{multline*}
For every $i = 1 \vvirg m$, $J_1(v_0 + \zeta_i v_1) \boxtimes T(\Gamma',\bfm') = T(\Gamma',\bfm')$. By \cite[Corollary 3.7]{BDG23}, the annihilator of $T(\Gamma',\bfm')$ is the Lie algebra of the gauge subgroup 
\[
\fraksl(E_{12}) \times \bigtimes_{e \in \bfe(\Gamma')\setminus \{(1,2)\}} \fraksl(E_e) 
\]
embedded in $\frakgl(W_1) \times \frakgl(W_2) \times \frakgl(W_3) \ttimes \frakgl(W_d)$, with $\frakgl(E_e)$ acting on the bond spaces $E_e \otimes E_e^*$ which appear as factors in $W_1 \ootimes W_d$. 

The summands $\bigtimes_{e \in \bfe(\Gamma')} \fraksl(E_e)$ of the annihilator are the same for each of the $m$ terms $J_1(v_0 + \zeta_i v_1) \boxtimes T(\Gamma',\bfm)$. The summand $\fraksl(E_{12})$ is embedded in $m$ different ways in $\frakgl(W_1) \ttimes \frakgl(W_d)$, according to the direct sum decomposition $W_i = W_i' \otimes \bbC^m = W_i'^{\oplus m}$, for $i=1,2$. Therefore \autoref{lemma: direct sum ann} provides an inclusion 
\[
\ann( A(\bfzeta)) \subseteq ( \textstyle \bigoplus_{i=1}^m \fraksl(E_{12})) \oplus \left(\bigtimes_{e \in \bfe(\Gamma')} \fraksl(E_e)\right).
\]
The reverse inclusion can be verified directly. Therefore 
\[
\dim \calTNS^\Gamma_{\bfm , \bfn} = m + \sum_{j=1}^d \dim \frakgl(W_j) - \dim \ann(T)
\]
which yields the claimed formula for the defect.
\end{proof}

The proof of \autoref{augmentation:theorem} gives an explicit description of the fiber of the map $\tilde{\Phi}$, that we record here:
\begin{corollary}\label{cor:augmentation_stabilizer}
In the setting of \autoref{augmentation:theorem}, let 
\begin{align*}
\tilde{\Phi} : \bbC^m \times \GL(W_1) \ttimes \GL(W_d) &\to V_0 \otimes W_1 \ootimes W_d \\
 (\bfzeta, X_1 \vvirg X_d)  & \mapsto ( X_1 \ootimes X_d) (A(\bfzeta)).
\end{align*}
where $A(\bfz) = ( \bigoplus_{i=1}^m J_1(v_0 + \zeta_i v_1)) \boxtimes T(\Gamma',\bfm')$.

Then the image of $\tilde{\Phi}$ is dense in the tensor network variety $\calTNS^{\Gamma}_{\bfm,\bfn}$. Moreover, the fiber of $A(\bfzeta)$ is the stabilizer of $A(\bfzeta)$ under the action of $\GL(W_1) \ttimes \GL(W_d)$ and its Lie algebra is 
\[
\left(\textstyle \bigoplus_{i=1}^m \fraksl_{m_{12}, i}\right) \oplus \left(\textstyle \bigoplus_{e \in \bfe(\Gamma')} \fraksl(E_e)\right)
\]
where $\fraksl_{m_{12}, i}$ is a copy of $\fraksl_{m_{12}}$ acting on the $i$-th summand of $W_s \simeq W_{s'}^{\oplus m}$ with $s = 1,2$ and the other summands act on the corresponding pair of bond spaces.
\end{corollary}

\begin{example}\label{ex:path_augmentation}
Let $\Gamma'$ be the path $1-2-3$ with bond dimensions
$m_{12}=r\ge 2$ and $m_{23}=s\ge 2$, and take the critical physical dimensions
$n'_1=r$, $n'_2=rs$, $n'_3=s$. 
Let $\Gamma$ be obtained by attaching a $2-2-2$ corner at the edge $(1,2)$,
i.e. adding a vertex $0$ and edges $(0,1),(0,2)$ with $m_{01}=m_{02}=2$ and $n_0=2$,
and set $n_1=2n'_1=2r$, $n_2=2n'_2=2rs$, $n_3=n'_3=s$.
Then \autoref{augmentation:theorem} yields
\[
\delta^\Gamma_{\bfm,\bfn}=\operatorname{fiberdefect}^\Gamma_{\bfm,\bfn}=r^2-1,
\]
and a generic $T\in \calTNS^\Gamma_{\bfm,\bfn}$ admits the normal form
\[
T \sim \bigl(J_1(z_1)\boxplus J_1(z_2)\bigr)\boxtimes T(\Gamma',\bfm'),
\quad z_1\neq z_2\in V_0\simeq \bbC^2.
\]
\end{example}

The characterization of the fiber in \autoref{cor:augmentation_stabilizer} allows us to obtain a generalization of \autoref{augmentation:theorem}, where we allow multiple augmentations. 

\begin{corollary}\label{cor:defect_growth_many_pins}
Let $\Gamma'=(\bfv(\Gamma'),\bfe(\Gamma'))$ be a network with vertex set $\bfv(\Gamma')=[d]=\{1,\dots,d\}$ and let $\bfm'$ be an assignment of bond dimensions on $\Gamma'$. Let $\{ i_1 \vvirg i_s, j_1 \vvirg j_s\} \subseteq \bfv(\Gamma')$ be set of $2s$ vertices. Let $\Gamma$ be the augmentation of $\Gamma'$ at nodes $\{i_r,j_r\}$ for $r = 1 \vvirg s$ with additional vertices $0_1 \vvirg 0_r$. Let $\bfm$ be an assignment of bond dimensions on $\Gamma$ with 
\[
m_{e} = m'_e \text{ for all } e \in \bfe(\Gamma'), \qquad m_{0_r, i_r} = m_{0_r j_r} = m_r \text{ for all } r \in [s].
\]
Let $\bfn$ be an assignment of physical dimension on $\Gamma$ where, for $r \in [s]$, $n_{0_r} = 2$ , and for every $k \neq 0_1 \vvirg 0_s$, $n_k=\prod_{\substack{e\ni k}} m_e$ is the critical dimension. Then $(\Gamma,\bfm,\bfn)$ is defective and
\[
\delta^{\Gamma}_{\bfm,\bfn} = \operatorname{fiberdefect}^{\Gamma}_{\bfm,\bfn} = \sum_{r=1}^s(m_r-1)(m_{i_rj_r}^2-1).
\]
Moreover, the generic element of $\calTNS^{\Gamma}_{\bfm,\bfn}$ has a normal form
\[
T = \textstyle \left(\bigoplus_{k=1}^{m_1} J_1(z_{1,k})\right) \boxtimes \cdots \boxtimes \left(\bigoplus_{k=1}^{m_s} J_1(z_{s,k})\right) \boxtimes T(\Gamma',\bfm')
\]
with factors suitably tensored together.
\end{corollary}
\begin{proof}
The normal form is obtained in a way similar to \autoref{augmentation:theorem}. Since the dimensions at vertices $k \neq 0_r$ for $r = 1 \vvirg s$ are critical, the tensor network variety may be regarded as a subvariety of 
\[
V_{0_1} \otimes V_{0_2} \ootimes V_{0_s} \otimes W_1 \ootimes W_d
\]
and its generic element, up to the action of $\GL(W_1) \ttimes \GL(W_d)$ is uniquely determined by generic $X_{0_1} \vvirg X_{0_s}$ with $X_{0_r} \in \Hom(W_{0_r}, V_{0_r}) = V_{0,r} \otimes E_{0_r i_r}^* \otimes E_{0_r j_r}^*$ for $r = 1 \vvirg s$. Each $X_{0_r}$ is a generic square pencil of size $m_r$, so by the Kronecker classification, we have $X_{0_r} = \bigoplus_{k=1}^{m_r} J_1(z_{r,k})$ for generic $z_{r,k} \in V_{0_r}$. Contracting each $X_{0_r}$ with the graph tensor $T(\Gamma,\bfm)$, we obtain the normal form 
\[
T = \textstyle \left(\bigoplus_{k=1}^{m_1} J_1(z_{1,k})\right) \boxtimes \cdots \boxtimes \left(\bigoplus_{k=1}^{m_s} J_1(z_{s,k})\right) \boxtimes T(\Gamma',\bfm').
\]
Write $\bfzeta_1 \vvirg \bfzeta_s$ for vectors in $\bbC^m$ determining the generalized eigenvalues of the pencils and write $A(\bfzeta_1 \vvirg \bfzeta_s)$ for the tensor in the normal form above. We obtain a parametrization of an open subset of $\calTNS^{\Gamma}_{\bfm,\bfn}$ via the map
\begin{align*}
\tilde{\Phi} : \bbC^{m_1} \times \cdots \times \bbC^{m_s} \times \GL(W_1) \ttimes \GL(W_d) &\to V_{0_1} \otimes \cdots \otimes V_{0_s} \otimes W_1 \ootimes W_d \\
 (\bfzeta_1, \vvirg, \bfzeta_s, X_1 \vvirg X_d)  & \mapsto ( X_1 \ootimes X_d) (A(\bfzeta_1, \vvirg, \bfzeta_s)).
\end{align*}
The dimension of the fiber of $\tilde{\Phi}$ determines the dimension of the tensor network variety. The same argument as in the proof of \autoref{augmentation:theorem} shows that for a generic $T \in \calTNS^{\Gamma}_{\bfm,\bfn}$, every element in the fiber $\tilde{\Phi}^{-1}(T)$ has the same components $\bfzeta_1 \vvirg \bfzeta_s$. Therefore, the fiber $\tilde{\Phi}^{-1}(T)$ is isomorphic, as a variety, to the stabilizer of $(A(\bfzeta_1, \vvirg, \bfzeta_s))$ under the action of $\GL(W_1) \ttimes \GL(W_d)$. This guarantees that the dimension of the fiber equals the dimension of the annihilator
\[
\ann_{\frakgl(W_1) \ttimes \frakgl(W_d)} (A(\bfzeta_1, \vvirg, \bfzeta_s)).
\] 
This is computed using induction on $s$ and \autoref{lemma: ann for block sum pencils}. We characterize the annihilator as
\begin{align*}
\ann_{\frakgl(W_1) \ttimes \frakgl(W_d)} &(A(\bfzeta_1, \vvirg, \bfzeta_s)) = \\ &=\left(\textstyle \bigoplus_{k=1}^{m_1} \fraksl_{m_{i_1 j_1}, k}\right) \oplus \cdots \oplus \left(\textstyle \bigoplus_{k=1}^{m_s} \fraksl_{m_{i_s j_s}, k}\right) \oplus \left(\textstyle \bigoplus_{e \in \bfe(\Gamma')} \fraksl(E_e)\right)
\end{align*}
with the action described in \autoref{cor:augmentation_stabilizer}.

If $s = 0$, the claim is trivial: in this case $A(\bfzeta_1 \vvirg \bfzeta_s) = T(\Gamma',\bfm')$ and the annihilator is the Lie algebra of the gauge group. 

If $s \ge 1$, we may write
\[
A(\bfzeta_1 \vvirg \bfzeta_s) = \left(\bigoplus_{k=1}^{m_1} J_1(z_{1,k})\right) \boxtimes \tilde{A}(\bfzeta_2 \vvirg \bfzeta_s)
\]
where $\tilde{A}(\bfzeta_2 \vvirg \bfzeta_s)$ is the tensor obtained by augmenting $\Gamma'$ at nodes $\{i_r,j_r\}$ for $r = 2 \vvirg s$. As in \autoref{augmentation:theorem}, we apply \autoref{lemma: ann for block sum pencils} iteratively. The spaces $M_1,M_2$ are trivial, because, as in \autoref{augmentation:theorem} and \autoref{thm: triangle mm1}, regarding $A(\bfzeta_1, \vvirg, \bfzeta_s)$ as a pencil, the generalized eigenvalues of the different Jordan blocks are distinct: this is guaranteed by the fact that the vertices $i_1,j_1$ are distinct from any $i_k, j_k$ for $k \geq 2$.

By \autoref{lemma: ann for block sum pencils}, we obtain that the annihilator decomposes as a direct sum of the annihilators of the different blocks summands $J_1(z_{1,k}) \boxtimes \tilde{A}(\bfzeta_2 \vvirg \bfzeta_s)$, which are all isomorphic to  ${A(\bfzeta_2 \vvirg \bfzeta_s)}$. By the induction hypothesis, we conclude that 
\begin{multline*}
\ann_{\frakgl(W_1) \ttimes \frakgl(W_d)} (A(\bfzeta_1, \vvirg, \bfzeta_s)) = \\
\left(\textstyle \bigoplus_{k=1}^{m_1} \fraksl_{m_{i_1 j_1}, k}\right) \oplus \cdots \oplus \left(\textstyle \bigoplus_{k=1}^{m_s} \fraksl_{m_{i_s j_s}, k}\right) \oplus \left(\textstyle \bigoplus_{e \in \bfe(\Gamma')} \fraksl(E_e)\right),
\end{multline*}
yielding the desired formula for the $\dim \calTNS^{\Gamma}_{\bfm,\bfn}$.
\end{proof}

\section*{Acknowledgments}
AB is a member of GNSAGA (INdAM).
AB has been partially supported  by European Union's HORIZON–MSCA-2023-DN-JD programme under the Horizon Europe (HORIZON) Marie Skłodowska-Curie Actions, grant agreement 101120296 (TENORS).
AB has been partially funded by the European Union under NextGeneration EU. PRIN 2022 Prot. n. 2022ZRRL4C 004. Views and opinions expressed are however those of the authors only and do not necessarily reflect those of the European Union or European Commission. Neither the European Union nor the granting authority can be held responsible for them.

We acknowledge the TensorDec Laboratory of the Department of Mathematics of the University of Trento, of which AB is currently a member, for helpful discussions. Part of this work was done while the authors were visiting the Simons Institute for the Theory of Computing, during the program \emph{Complexity and Linear Algebra} in Fall 2025: we thank the Simons Institute for its support and for providing an inspiring research environment. Finally, we are grateful to V. Lysikov, G. Ottaviani and J. Zuiddam for helpful discussions during the development of this work.

\bibliographystyle{alphaurl}
\bibliography{bibTNS.bib}

\end{document}